\documentclass[11pt,a4paper,reqno]{amsart}
\usepackage{color}
\usepackage{amsfonts,amsmath,amssymb,amsxtra,url,float,rotating} % 论文初始格式设定
\usepackage[colorlinks,linkcolor=RoyalBlue,anchorcolor=Periwinkle,citecolor=Orange,urlcolor=Green]{hyperref} % 超链颜色设定
\usepackage[usenames,dvipsnames]{xcolor} % 颜色宏包调用
\usepackage{enumitem}
\usepackage{geometry} 
\usepackage{graphicx} % 图片宏包调用, 调用此宏包后将被允许插入png 格式图片
\usepackage{subfigure} % figure环境下的补充宏包, 允许在figure环境下对自图片进行分栏
\DeclareMathAlphabet{\mathpzc}{OT1}{pzc}{m}{it} %数学符号字体的设置
\usepackage{systeme}
\usepackage{tikz}
\usetikzlibrary{matrix}
\usetikzlibrary{trees}
\usepgflibrary{decorations.pathmorphing}
\usetikzlibrary{decorations.pathmorphing}
\usetikzlibrary{decorations.markings}
 \tikzset{
  on each segment/.style={
    decorate,
    decoration={
      show path construction,
      moveto code={},
      lineto code={
        \path [#1]
        (\tikzinputsegmentfirst) -- (\tikzinputsegmentlast);
      },
      curveto code={
        \path [#1] (\tikzinputsegmentfirst)
        .. controls
        (\tikzinputsegmentsupporta) and (\tikzinputsegmentsupportb)
        ..
        (\tikzinputsegmentlast);
      },
      closepath code={
        \path [#1]
        (\tikzinputsegmentfirst) -- (\tikzinputsegmentlast);
      },
    },
  },
  mid arrow/.style={postaction={decorate,decoration={
        markings,
        mark=at position .5 with {\arrow[#1]{stealth}}
      }}},
}
\usetikzlibrary{arrows}

\usetikzlibrary{patterns}
\usetikzlibrary{shadings} % 调用tikz作图指令包
\usepackage[all]{xy} % 交换图箭图宏包调用, 允许在xymatrix环境下进行交换图表编译
\usepackage{mathdots} % 省略记号宏包调用, 调用此宏包后将被允许使用竖直、倾斜等多种省略符号
\usepackage{yhmath} %
\usepackage{dsfont} % 可以调用mathds指令
\usepackage{cite}
\usepackage{mathrsfs} % mathscr
\usepackage{multicol} %用于实现在同一页中实现不同的分栏
\usepackage{changepage}
\numberwithin{figure}{section}

\usepackage{marginnote} % 可使用边注\marginnote{}指令
\usepackage{graphicx} % 可使用\figure环境
\usepackage{multicol} %用于实现在同一页中实现不同的分栏
\newtheorem{theorem}{Theorem}[section]
\newtheorem{lemma}[theorem]{Lemma}
\newtheorem{corollary}[theorem]{Corollary}
\newtheorem{main theorem}[theorem]{Main Theorem}
\newtheorem{proposition}[theorem]{Proposition}
\newtheorem{definition}[theorem]{Definition}

\newtheorem{remark}[theorem]{Remark}
\newtheorem{example}[theorem]{Example}

\usetikzlibrary{arrows}

\numberwithin{equation}{section}

\def\<{\langle} % 左尖括号指令省略为"\<"
\def\>{\rangle} % 右尖括号指令省略为"\>"
\newcommand{\Pic}{Figure\ }
\newcommand{\modcat}{\mathsf{mod}}

\newcommand{\ind}{\mathsf{ind}}
\newcommand{\str}{\mathsf{Str}}
\newcommand{\band}{\mathsf{Band}}
\newcommand{\kk}{\mathds{k}} % mathds双写字体的"k", 用于描述域
\newcommand{\Q}{\mathcal{Q}} % mathcal花体的"Q", 用于描述箭图
\newcommand{\I}{\mathcal{I}} % mathcal花体的"I", 用于描述admissible 理想
\newcommand{\M}{\mathds{M}}
\newcommand{\tA}{\tilde{\mathbb{A}}} % Nakayama tilde A
\newcommand{\A}{\mathbb{A}} % Nakayama A
\renewcommand{\t}{\mathfrak{t}}
\newcommand{\s}{\mathfrak{s}}
\newcommand{\lu}{\mathrm{lu}}
\newcommand{\ru}{\mathrm{ru}}
\newcommand{\ld}{\mathrm{ld}}
\newcommand{\rd}{\mathrm{rd}}
\renewcommand{\d}{\mathrm{d}}
\renewcommand{\u}{\mathrm{u}}
\newcommand{\inner}{\mathrm{in}}
\newcommand{\out}{\mathrm{out}}
\newcommand{\id}{\mathrm{inj.dim}}

\newcommand{\coker}{\mathrm{coker}}
\newcommand{\Left}{\mathrm{L}}
\newcommand{\Right}{\mathrm{R}}
\newcommand{\EE}{\mathds{E}}

\renewcommand{\top}{\mathrm{top}}
\newcommand{\soc}{\mathrm{soc}}
\newcommand{\frakR}{\mathfrak{R}}
\newcommand{\ELIS}{\mathrm{ELIS}}

\newcommand{\lmpset}{\mathrm{LMP}}

\newcommand{\length}{\mathfrak{l}}

\newcommand{\To}[1]{\mathop{-\!\!-\!\!\!\to}\limits^{#1}}
\def\defines{\it}
\newcommand{\checks}[1]{{\color{black}{#1}}}
\newcommand{\pnotation}{\checks{\widehat{p}}}
\newcommand{\nota}{\widehat{p}}

\begin{document}
%=========================================================

\title[On the Gorensteiness of string algebras]{On the Gorensteiness of string algebras}
%\thanks{$^{\ast}$Corresponding author.}
\thanks{MSC2020: 16G10, 16E45.}
\thanks{Key words: string algebra; self-injective dimension; effective intersecting relation; Gorenstein algebra}
\author{Houjun Zhang}
\address{Houjun Zhang, School of Science, Nanjing University of Posts and Telecommunications, Nanjing 210023, P. R. China}
\email{zhanghoujun@njupt.edu.cn}
\author{Dajun Liu}
\address{School of Mathematics-Physics and Finance, Anhui Polytechnic University, Wuhu, China 241000, P. R. China}
\email{liudajun@ahpu.edu.cn}
\author{Yu-Zhe Liu}
\address{School of Mathematics and Statistics, Guizhou University, Guiyang 550025, P. R. China}
\email{yzliu3@163.com}

\begin{abstract}
In this paper,  we give a description of the self-injective dimension of string algebras
and obtain a necessary and sufficient condition for a string algebra to be Gorenstein.
\end{abstract}

\maketitle
%=========================================================
\section{Introduction}
Throughout this paper $\kk$ is a field. Let $A$ be a finite-dimensional $\kk$-algebra. Recall that $A$ is called a Gorenstein algebra if $A$ has finite injective dimension both as a left and a right $A$-module. Gorenstein algebra \checks{arises} from commutative ring theory and \checks{plays} a central role in \checks{the} representation theory of finite-dimensional algebras. Being Gorenstein algebra has nice properties, such as:
if $A$ is Gorenstein, then the bounded homology categories $K^{b}(\mathrm{proj} A)$ of projectives and $K^{b}(\mathrm{inj} A)$ of injectives are coincide. Moreover, Gorenstein algebra is preserved under derived equivalence \cite{H1991}.

It is well known that the self-injective algebras and the algebras of finite global dimension are Gorenstein. Recently, Gei\ss\ and Retein showed that gentle algebras are Gorenstein \cite{GR2005}; they also pointed out that skewed-gentle algebras \checks{considered} in \cite{Geisdela99} are Gorenstein if the field is not of characteristic 2. Gentle algebras introduced by Assem and Skowro\'{n}ski \cite{AS1987} form a particular class of string algebras. However, there are examples showed that string algebras are not Gorenstein. In this paper, we study when a string algebra is Gorenstein.

String algebras were introduced by Butler and Ringel in \cite{BR1987}. They \checks{gave} an explicit description of the finite-dimensional indecomposable modules and the irreducible maps between them and provided a method to draw the Auslander-Reiten quivers of string algebras. In order to study the Gorensteinness of string algebra, we first give a description of the self-injective dimension of string algebra by using the description of the finite-dimensional indecomposable modules in \cite{BR1987}.

Let $A$ be a finite-dimensional $\kk$-algebra.  We denote by $\id ({_AA})$ and $\id (A_A)$ the left injective dimension and the right injective dimension of $A$ respectively.
\checks{By \cite[Corollary 2.4]{GKK91}, we have that the finitistic dimension conjecture holds true for string algebras since any string algebra is monomial.
And so, the Gorenstein symmetry conjecture holds true for string algebras by \cite{Hu1995}.
In \cite[Proposition 6.10]{AR1991Nakayama}, Auslander and Reiten have proved that if $\id ({A_A})<\infty$, then $\id ({_AA})<\infty$ if and only if the finitistic dimension of $A$ is finite.
It follows that $\id({_AA}) = \id({A_A})$ since an algebra is a string algebra if and only if so is its opposite algebra. }
Thus, in this paper, we consider the right $A$-modules and the right injective dimension of $A$. For simplicity, we denote by $\id({A})=\id({A_A}) = \id({_AA})$.

Now assume that $\Q = (\Q_0, \Q_1, \s, \t)$ is a finite quiver, where $\s, \t: \Q_1 \to \Q_0$ are
two functions sending arrow $\alpha\in \Q_1$ to its starting point $\s(\alpha)$ and ending point $\t(\alpha)$.
Let $A=\kk\Q/\I$ be a string algebra (see Definition \ref{def:string}).
To compute the injective dimension of $A$, we define the intersecting relations for a string algebra.
Notice that the injective dimension of $A$ is determined by some intersecting relations.
Thus, we introduce ELIS the effective left (and right) intersecting relations
with respect to a projective module (see Definition \ref{def:ELIS}).
Denote by $P(v_0)$ the projective module of vertex $v_0$ in $\Q$.
Let $\{r_i\}_{i=1}^{n}=\{r_n,\ldots,r_1\}$ be an ELIS with respect to $P(v_0)$
and define $\length(\{r_i\}_{i=1}^n)=n$ as the length of $\{r_i\}_{i=1}^{n}$.
Furthermore, the length of $\ELIS(P(v_0))$ can be define as:
\[\length(\ELIS(P(v_0))) :=
\begin{cases}
\sup\limits_{\{r_i\}_{i=1}^{n} \in \ELIS(P(v_0))} \length(\{r_i\}_{i=1}^{n}), & \text{ if $\ELIS(P(v_0)) \ne \varnothing$; } \\
\quad\quad\quad\quad\quad 0, &\text{ if $\ELIS(P(v_0)) = \varnothing$. }
\end{cases}\]
Then we obtain the main result of this paper:
\begin{theorem}
Let $A=\kk\Q/\I$ be a string algebra. If $A$ is not self-injective, then
\[ \id A = \sup_{v_0\in \Q_0} \length(\ELIS(P(v_0))) + 1. \]
\end{theorem}

According to the above Theorem,  we immediately obtain a necessary and sufficient condition for a string algebra to be Gorenstein.

\begin{corollary}
Let $A=\kk\Q/\I$ be a string algebra. Then $A$ is Gorenstein if and only if all $\mathrm{ELISs}$ of $P(v_0)$ have a supremum for any vertex $v_0$ in $\Q_0$.
\end{corollary}

This paper is organized as follows. In Section 2, we recall some preliminaries on string algebras. Section 3 is devoted to study the cosyzygies of indecomposable modules for string algebras. Based on the injective decomposition of projective modules, we define the effective left and right intersecting relations with respect to a projective module. In the last Section, we give the proof of the main results.

Let $\Q$ be a finite quiver. For arbitrary two arrows $\alpha$ and $\beta$ of the quiver $\Q$,
if $\t(\alpha)=\s(\beta)$, then the composition of $\alpha$ and $\beta$ is denoted by $\alpha\beta$. Assuming $A$ is a finite-dimensional algebra, we denote by $\modcat(A)$ the category of right $A$-modules.
For arbitrary two modules $M$ and $N$, if $M$ is a direct summand of $N$, then we denote by $M\le_{\oplus} N$.
Moreover, we denote by $S(v)$, $P(v)$, and $E(v)$ the simple, indecomposable projective and indecomposable injective modules corresponding to the vertex $v\in \Q_0$, respectively.

\section{String algebras}
In this section, we recall the definition and some properties of string algebras. We refer the reader to \cite{BR1987} for more detail. Throughout this paper, we always assume that $\Q$ is a finite connected quiver.

\begin{definition}\rm
Let $\Q$ be a finite quiver and $\I$ an admissible ideal of $\kk\Q/\I$.
The pair $(\Q, \I)$ is said to be a {\defines string quiver} if it satisfies the following conditions:
\begin{itemize}
  \item[(1)] any vertex of $\Q$ is the source and target of at most two arrows;

  \item[(2)] for each arrow $\beta$, there is at most one arrow $\gamma$ such that $\beta\gamma\notin\I$;

  \item[(3)] for each arrow $\beta$, there is at most one arrow $\alpha$  such that $\alpha\beta\notin\I$;

  \item[(4)] $\I$ is generated by paths of length \checks{greater} than or equal to $2$.
\end{itemize}
\end{definition}

\begin{definition}\rm \label{def:string}
Let $(\Q, \I)$ be a string quiver. A finite-dimensional $\kk$-algebra $A$ is called a {\defines string algebra}
if it is Morita equivalent to $\kk\Q/\I$.
\end{definition}

%Note that the bound quiver $(\Q,\I)$ of any string algebra $A=\kk\Q/\I$ is a finite quiver, that is, the number of vertices and arrows is finite.
Let $A=\kk\Q/\I$ be a string algebra and $v_0 \in \Q_0$. By the definition, we obtain that there are eight types for $v_0$.
\begin{figure}[H]
\begin{tikzpicture}[scale=1]
\draw (0,0) node{$v_0$};
\draw (-1, 1) node{$v_{\lu}$};
\draw ( 1, 1) node{$v_{\ru}$};
\draw ( 1,-1) node{$v_{\rd}$};
\draw (-1,-1) node{$v_{\ld}$};
\draw[->] (-0.8, 0.8) -- (-0.2, 0.2); \draw (-0.4, 0.4) node[left]{$a_{\lu,0}$};
\draw[->] ( 0.8, 0.8) -- ( 0.2, 0.2); \draw ( 0.4, 0.4) node[right]{$a_{\ru,0}$};
\draw[->] ( 0.2,-0.2) -- ( 0.8,-0.8); \draw ( 0.4,-0.4) node[right]{$a_{0,\rd}$};
\draw[->] (-0.2,-0.2) -- (-0.8,-0.8); \draw (-0.4,-0.4) node[left]{$a_{0,\ld}$};
\draw (0,-1.5) node{Type $(2^{\inner}, 2^{\out})$};
\end{tikzpicture}
\ \
\begin{tikzpicture}[scale=1]
\draw (0,0) node{$v_0$};
\draw (-1, 1) node{$v_{\lu}$};
\draw ( 1, 1) node{$v_{\ru}$};
\draw ( 0,-1) node{$v_{\d}$};
\draw[->] (-0.8, 0.8) -- (-0.2, 0.2); \draw (-0.4, 0.4) node[left]{$a_{\lu,0}$};
\draw[->] ( 0.8, 0.8) -- ( 0.2, 0.2); \draw ( 0.4, 0.4) node[right]{$a_{\ru,0}$};
\draw[->] ( 0.0,-0.2) -- ( 0.0,-0.8); \draw ( 0.0,-0.4) node[right]{$a_{0,\d}$};
\draw (0,-1.5) node{Type $(2^{\inner}, 1^{\out})$};
\end{tikzpicture}
\ \ \ \
\begin{tikzpicture}[scale=1]
\draw (0,0) node{$v_0$};
\draw ( 0, 1) node{$v_{\u}$};
\draw ( 1,-1) node{$v_{\rd}$};
\draw (-1,-1) node{$v_{\ld}$};
\draw[->] ( 0.0, 0.8) -- ( 0.0, 0.2); \draw ( 0.0, 0.4) node[left]{$a_{\u,0}$};
\draw[->] ( 0.2,-0.2) -- ( 0.8,-0.8); \draw ( 0.4,-0.4) node[right]{$a_{0,\rd}$};
\draw[->] (-0.2,-0.2) -- (-0.8,-0.8); \draw (-0.4,-0.4) node[left]{$a_{0,\ld}$};
\draw (0,-1.5) node{Type $(1^{\inner}, 2^{\out})$};
\end{tikzpicture}
\ \ \ \
\begin{tikzpicture}[scale=1]
\draw (0,-1) node{$v_0$};
\draw (-1, 0) node{$v_{\lu}$};
\draw ( 1, 0) node{$v_{\ru}$};
\draw[->] (-0.8, -0.2) -- (-0.2, -0.8); \draw (-0.4, -0.6) node[left]{$a_{\lu,0}$};
\draw[->] ( 0.8, -0.2) -- ( 0.2, -0.8); \draw ( 0.4, -0.6) node[right]{$a_{\ru,0}$};
\draw (0,-1.5) node{Type $(2^{\inner}, 0^{\out})$};
\end{tikzpicture}
\end{figure}

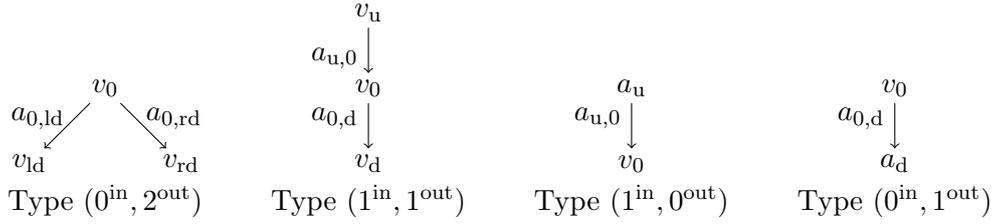
\begin{figure}[H]
\begin{tikzpicture}[scale=1]
\draw (0,0) node{$v_0$};
\draw ( 1,-1) node{$v_{\rd}$};
\draw (-1,-1) node{$v_{\ld}$};
\draw[->] ( 0.2,-0.2) -- ( 0.8,-0.8); \draw ( 0.4,-0.4) node[right]{$a_{0,\rd}$};
\draw[->] (-0.2,-0.2) -- (-0.8,-0.8); \draw (-0.4,-0.4) node[left]{$a_{0,\ld}$};
\draw (0,-1.5) node{Type $(0^{\inner}, 2^{\out})$};
\end{tikzpicture}
\ \ \ \
\begin{tikzpicture}[scale=1]
\draw (0,0) node{$v_0$};
\draw ( 0, 1) node{$v_{\u}$};
\draw ( 0,-1) node{$v_{\d}$};
\draw[->] ( 0.0, 0.8) -- ( 0.0, 0.2); \draw ( 0.0, 0.4) node[left]{$a_{\u,0}$};
\draw[->] (-0.0,-0.2) -- (-0.0,-0.8); \draw (-0.0,-0.4) node[left]{$a_{0,\d}$};
\draw (0,-1.5) node{Type $(1^{\inner}, 1^{\out})$};
\end{tikzpicture}
\ \ \ \
\begin{tikzpicture}[scale=1]
\draw (0,-1) node{$v_0$};
\draw ( 0, 0) node{$a_{\u}$};
\draw[->] ( -0.0, -0.2) -- ( -0.0, -0.8); \draw ( 0.0,- 0.4) node[left]{$a_{\u,0}$};
\draw (0,-1.5) node{Type $(1^{\inner}, 0^{\out})$};
\end{tikzpicture}
\ \ \ \
\begin{tikzpicture}[scale=1]
\draw (0,0) node{$v_0$};
\draw ( 0,-1) node{$a_{\d}$};
\draw[->] (-0.0,-0.2) -- (-0.0,-0.8); \draw ( 0.0,-0.4) node[left]{$a_{0,\d}$};
\draw (0,-1.5) node{Type $(0^{\inner}, 1^{\out})$};
\end{tikzpicture}
\caption{The vertex $v_0$ in the bound quiver of string algebra. }
\label{fig:vertex}
\end{figure}

For any arrow $a\in \Q_1$, we denote by $a^{-1}$ the formal inverse of $a$. Then $\s(a^{-1})=\t(a)$ and $\t(a^{-1})=\s(a)$.
We denote by $\Q_1^{-1}:=\{a^{-1}\mid a\in \Q_1\}$ the set of all formal inverses of arrows. Any path $p=a_1a_2\cdots a_\ell$ in $(\Q, \I)$ naturally provides a formal inverse path $p^{-1} = a_{\ell}^{-1}a_{\ell-1}^{-1}\cdots a_1^{-1}$ of $p$. For any trivial path $e_v$ corresponding to $v\in \Q_0$, we define $e_v^{-1} = e_v$.

\begin{definition}\rm
Let $A=\kk\Q/\I$ be a string algebra.

(1) A {\defines string} over $(\Q, \I)$ is a sequence $s=(p_1, p_2, \ldots, p_n)$ such that:
\begin{itemize}
  \item %[(St1)]
    for any $1\le i\le n$, $p_i$ or $p_i^{-1}$ is a path in $(\Q, \I)$ \checks{whose length is greater than or equal to $1$};

  \item %[(St2)]
    if $p_i$ is a path, then $p_{i+1}$ is a formal inverse path;

  \item %[(St3)]
    $\t(p_{i})=\s(p_{i+1})$ holds for all $1\le i\le n-1$.
\end{itemize}

\checks{In particular, if $n=0$, that is, $s$ is an empty set, then it is called a {\defines trivial string}.
A trivial string is always used to describe the module $0$.
A string $s$ is called a {\defines directed string} if $s$ is either a trivial string or a path.}

(2) A {\defines band} $b=(p_1, p_2, \ldots, p_n)$ is a string such that:
  \begin{itemize}
    \item %[(Ba1)]
      $\t(p_n)=\s(p_1)$ and $p_np_1\notin \I$;

    \item %[(Ba2)]
      $b$ is not a non-trivial power of some string, i.e.,
      there is no string $s$ such that $b=s^m$ for some $m\ge 2$.
  \end{itemize}
\end{definition}

Two strings $s$ and $s'$ are called {\defines equivalent} if $s'=s$ or $s'=s^{-1}$;
Two bands $b=\alpha_1\cdots\alpha_n$ and $b'=\alpha_1'\cdots\checks{\alpha_n'}$ are called {\defines equivalent} if $b[t]=b'$ or $b[t]^{-1}=b'$, where
\[\checks{
b[t]=\begin{cases}
 \alpha_{1+t}\cdots\alpha_{n}\alpha_1\cdots\alpha_{1+t-1}, & 1\le t\le n-1; \\
 \alpha_1\cdots\alpha_n = b, & t=0.
\end{cases}
}\]
We denote by $\str(A)$ the set of all equivalent classes of strings and by $\band(A)$ the set of all equivalent classes of bands on the bound quiver of $A$, respectively. In \cite{BR1987}, Butler and Ringel showed that all indecomposable modules over a string algebra can be described by strings and bands. They proved that there is a bijection
\[ \M: \str(A) \cup (\band(A)\times\mathscr{J}) \to \ind(\modcat(A)), \]
where $\mathscr{J}$ is the set of all indecomposable $\kk[x,x^{-1}]$-modules.
Usually, if $\M^{-1}(N)$ is a string, then we say $N$ is a {\defines string module}. If $\M^{-1}(N)$ is a band with some pair $(n,\lambda)$, we say it is a {\defines band module}. The original definition of string and band modules over string algebra can be referred to \cite{BR1987}.

Now let $s=a^{-1}_{r}\cdots a^{-1}_2a^{-1}_1 b_1b_2\cdots b_t$ be a string with $r,t\ge 0$, $a^{-1}_1, \ldots, a^{-1}_r \in \Q_1^{-1}$ and $b_1,\ldots,b_t \in \Q_1$. If it satisfies the following conditions:
\begin{itemize}
  \item[(1)] $\t(a^{-1}_1)=\s(b_1)$,
  \item[(2)] for any $\alpha \in \Q_1$ with $\t(a_r)=\s(\alpha)$, $a_{r'}a_{r'+1}\cdots a_r\alpha \in \I$ for some $1\le r'\le r$, and
  \item[(3)] for any $\beta \in \Q_1$ with $\t(b_t)=\s(\beta)$, $b_{t'}b_{t'+1}\cdots b_t\beta \in \I$ for some $1\le t'\le t$.
\end{itemize}
Then $\M(s)$ is an indecomposable projective $A$-module, we call that $s$ is a projective string in this case.
Dually, we can define any indecomposable injective string.

\checks{\begin{example} \label{exp:str alg} \rm
Let $A$ be the path algebra given by the following quiver
\begin{center}
\begin{tikzpicture}[scale=0.5]
\draw[->][rotate=  0+10] (2,0) arc (0:25:2);
\draw[->][rotate= 45+10] (2,0) arc (0:25:2);
\draw[->][rotate= 90+10] (2,0) arc (0:25:2);
\draw[->][rotate=135+10] (2,0) arc (0:25:2);
\draw[->][rotate=180+10] (2,0) arc (0:25:2);
\draw[->][rotate=225+10] (2,0) arc (0:25:2);
\draw[->][rotate=270+10] (2,0) arc (0:25:2);
\draw[->][rotate=315+10] (2,0) arc (0:25:2);
\draw[rotate=  0] (2,0) node{$1$}; \draw[rotate=  0] (2.25,0.85) node{$a$};
\draw[rotate= 45] (2,0) node{$2$}; \draw[rotate= 45] (2.25,0.85) node{$b$};
\draw[rotate= 90] (2,0) node{$3$}; \draw[rotate= 90] (2.25,0.85) node{$c$};
\draw[rotate=135] (2,0) node{$4$}; \draw[rotate=135] (2.25,0.85) node{$d$};
\draw[rotate=180] (2,0) node{$5$}; \draw[rotate=180] (2.25,0.85) node{$e$};
\draw[rotate=225] (2,0) node{$6$}; \draw[rotate=225] (2.25,0.85) node{$f$};
\draw[rotate=270] (2,0) node{$7$}; \draw[rotate=270] (2.25,0.85) node{$g$};
\draw[rotate=315] (2,0) node{$8$}; \draw[rotate=315] (2.25,0.85) node{$h$};
\draw[rotate= 45][->] (2.3,0) -- (4,0) node[right]{$9$};
\draw[rotate= 45] (2.5,0) node[right]{$x$};
\draw[rotate= 45+180][->] (2.3,0) -- (4,0) node[left]{$10$};
\draw[rotate= 45+180] (2.5,0) node[left]{$y$};
\draw[rotate= 45+5][->] (4.3,0) arc (0:160:4.3);
\draw (-3,3) node[left]{$z$};
% relations
\draw[rotate=  0][dotted] (2,0) to[out=135,in=45] (-2,0);
\draw[rotate= 90][dotted] (2,0) to[out=135,in=45] (-2,0);
\draw[rotate=180][dotted] (2,0) to[out=135,in=45] (-2,0);
\draw[rotate=270][dotted] (2,0) to[out=135,in=45] (-2,0);
\draw[rotate=405][dotted] (2,0) to[out=135,in=45] (-2,0);
\draw[rotate=585][dotted] (2,0) to[out=135,in=45] (-2,0);
\draw[dotted][rotate=45] (2,-1) -- (3,0);
\draw[dotted][rotate=45+180] (2,-1) -- (3,0);
\end{tikzpicture}
\end{center}
with $\I = \langle abcd, bcde, cdef, efgh, fgha, ghab, ax, ey \rangle$.
Then it is a string algebra. Let $P(2)$ be the projective module of vertex $2$,
 then the equivalent class of the string corresponding to $P(2)$ is
$\M^{-1}(P(2)) = \{z^{-1}x^{-1}bcd, d^{-1}c^{-1}b^{-1}xz\}$.
For simplicity, we do not discriminate between $z^{-1}x^{-1}bcd$ and $d^{-1}c^{-1}b^{-1}xz$.
\end{example}}

\section{The cosyzygy of indecomposable projective modules}
In this section, we give a description of cosyzygies of indecomposable projective modules for string algebras.
Let $A=\kk\Q/\I$ be a string algebra and $M$ an $A$-module.
Recall that the injective dimension $\id M$ of $M$ is less than or equal to $n$ if
there exists an exact sequence
\[ 0 \longrightarrow M \mathop{\longrightarrow}\limits^{f_0} E_0
\mathop{\longrightarrow}\limits^{f_1} E_1
\mathop{\longrightarrow}\limits^{f_2} \cdots
\mathop{\longrightarrow}\limits^{f_{n-1}} E_n
\longrightarrow 0, \]
where $E_i$ $(0\le i\le n)$ is an injective module. The {\defines self-injective dimension} of the finite-dimensional algebra $A$ is the injective dimension $\id A$ of $A_A \in \modcat A$.

For any module $M$ in $\modcat A$, assume that $e^M_0: M \to \EE_0(M)$ is the injective envelope of $M$.
Then $e^M_{i+1}: \EE_i(M) \to \EE_{i+1}(M)$ is induced by the injective envelope
$$e^{\coker(e^M_{i})}_0: \coker(e^M_{i}) \to \EE_0(\coker(e^M_{i})),$$
for any $i\ge 0$, where $\coker(e^M_{i}) = \EE_i(M)/\mathrm{im}(e^M_i)$.
We call $\coker(e^M_{i})$ the {\defines $(i+1)$th-cosyzygy} of $M$ and denote it by $\mho_{i+1}(M)$.

\subsection{The cosyzygies of indecomposable projective modules}
We want to compute the self-injective dimension of string algebra.
To do this, we need to study the cosyzygy of indecomposable projective module.

\begin{definition}\rm
A vertex $v_0$ of string quiver is said to be a {\defines relational vertex} if there are two arrows
$a$ and $b$ with $\t(a)=\s(b)=v_0$ such that $ab\in \I$;
otherwise, we say $v_0$ is a {\defines non-relational vertex}.
A relational vertex is said to be a {\defines strictly relational vertex}
if it is the following forms:
\begin{itemize}
  \item[(1)] the Type $(2^{\inner}, 2^{\out})$ in \Pic \ref{fig:2in2out} and at least one of $a_{1}d_{1}$ and $b_{1}c_{1}$
    is a subpath of some generator of $\I$;
  \item[(2)] the Type $(1^{\inner}, 2^{\out})$ and all the paths $\alpha\beta$ with $\t(\alpha)=\s(\beta)=v_0$ are belong to $\I$ or a subpath of some generator of $\I$.
\end{itemize}
For a vertex $v_0$ which is one of the Types \checks{$(2^{\inner}, 2^{\out})$ and $(1^{\inner}, 2^{\out})$}, if it is not strictly relational, then we say it is a {\defines gently relational vertex}.
\end{definition}

\begin{proposition} \label{prop:0-cosyzygy-I}
Let $A=\kk\Q/\I$ be a string algebra and $v_0$ be a vertex in $\Q$ of the Types \checks{$(2^{\inner}, 2^{\out})$ or $(1^{\inner}, 2^{\out})$}.
%% ========== 2025-2-22 00:13:38 ==========
%% delete Type $(2^{\inner}, 1^{\out})$
%% ========================================
Then
\begin{align}\label{formula:0-cosyzygy-I}
  \mho_1(P(v_0))\cong D_{\Left}\oplus D\oplus D_{\Right}
\end{align}
where $D_{\Left}$ and $D_{\Right}$ are indecomposable modules which \checks{correspond} to directed strings and $D$ is an indecomposable module whose socle is $S(v_0)$. Furthermore, $D$ is injective if and only if $v_0$ is a gently relational vertex.
\end{proposition}

\begin{proof}
\checks{We only prove for the Type $(2^{\inner}, 2^{\out})$, the other Type is similar.}
Assume \checks{that} $v_0$ is a vertex of the Type $(2^{\inner}, 2^{\out})$,
\checks{and the indecomposable projective module $P(v_0)$ is non-injective.
Then the string corresponding to} $P(v_0)$ is
\begin{center}
$c_{\ell_{\ld}}^{-1}\cdots c_2^{-1} c_1^{-1}d_1d_2\cdots d_{\ell_{\rd}}$
\end{center}
for \checks{some} $\ell_{\ld}, \ell_{\rd} \ge 1$. Thus, the injective envelope of $P(v_0)$ is
\noindent
\begin{center}
$e_0^{P(v_0)}: P(v_0) \to \EE(P(v_0)) \cong
E(v_{\ld}^{(\ell_{\ld})}) \oplus E(v_{\rd}^{(\ell_{\rd})})$,
\end{center}
where $E(v_{\ld}^{(\ell_{\ld})})$ and $E(v_{\rd}^{(\ell_{\rd})})$
are string modules respectively corresponding to strings
\begin{center}
$\wp_{\Left} \cdot c_{\ell_{\ld}}^{-1}\cdots c_2^{-1} c_1^{-1}b_1^{-1}b_2^{-1}\cdots b_{\ell_{\ru}}^{-1}$
and
$\wp_{\Right}\cdot d_{\ell_{\rd}}^{-1}\cdots d_2^{-1} d_1^{-1}a_1^{-1}a_2^{-1}\cdots a_{\ell_{\lu}}^{-1}$
\end{center}
for two paths $\wp_{\Left}$ and $\wp_{\Right}$ in $(\Q, \I)$ with $\t(\wp_{\Left}) = \t(c_{\ell_{\ld}})$ and $\t(\wp_{\Right}) = \t(d_{\ell_{\rd}})$ (see \Pic \ref{fig:2in2out}).
Then $$\mho_0(P(v_0)) \cong D_{\Left} \oplus D \oplus D_{\Right}, $$
where:
$\M^{-1}(D_{\Left}) = \widetilde{\wp}_{\Left}$,
$\M^{-1}(D) = a_{\ell_{\lu}} \cdots a_2a_1 b_1^{-1}b_2^{-1}\cdots b_{\ell_{\ru}}^{-1}$
($\ell_{\lu},\ell_{\ru} \ge 0$)
and $\M^{-1}(D_{\Right}) = \widetilde{\wp}_{\Right}$.
\checks{Notice that if $\ell_{\lu}=0$, then $\M^{-1}(D) = b_1^{-1}b_2^{-1}\cdots b_{\ell_{\ru}}^{-1}$;
if $\ell_{\ru}=0$, then $\M^{-1}(D) = a_{\ell_{\lu}} \cdots a_2a_1$;
and if $\ell_{\lu}=0$ and $\ell_{\ru}=0$, then $\M^{-1}(D)$ is the string of length zero corresponding to $v_0$.
}

\begin{figure}[htbp]
\centering
\includegraphics[width=17cm]{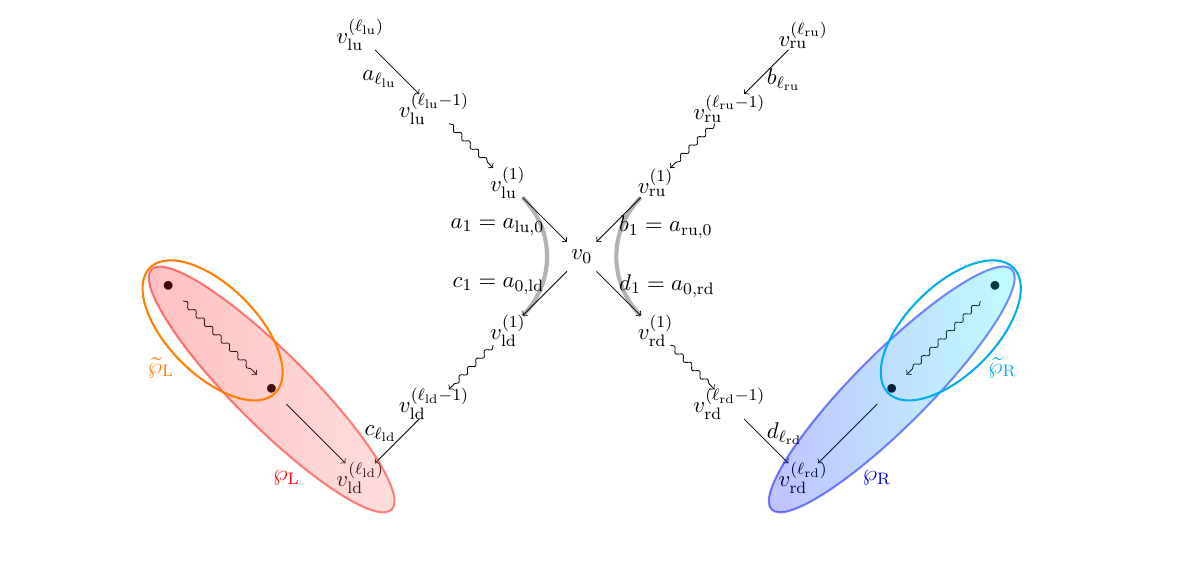}\\
\caption{The vertex $v_0$ of Type $(2^{\inner}, 2^{\out})$,
where $a_1c_1\in\I$ and $b_1d_1\in\I$}
\label{fig:2in2out}
\end{figure}

If $v_0$ is a strictly relational vertex, without loss of generality,
we assume that $b_1c_1$ is a subpath of some relation in $\I$.
Then there exists an integer $1\le t \le \ell_{\ld}+1$ such that
$b_{\ell_{\ru}+1}b_{\ell_{\ru}}\cdots b_2b_1 c_1 c_2 \cdots c_t$ is a generator of $\I$,
where $b_{\ell_{ru}+1}$ is an arrow ending at $v_{\ru}^{(\ell_{\ru})}$. Thus, $b_{\ell_{\ru}+1}b_{\ell_{\ru}}\cdots b_2b_1 \notin \I$,
this shows that $D$ is not injective. Now, if $v_0$ is a gently relational vertex, then we have
    \begin{itemize}
      \item[(1)] $v_{\ru}^{(\ell_{\ru})}$ is a source of $(\Q, \I)$ or there exists an integer $1\le x \le \ell_{\ru}+1$ such that
      $b_{\ell_{\ru}+1}b_{(\ell_{\ru})}\cdots b_x$ is a generator of $\I$;
      \item[(2)] $v_{\lu}^{(\ell_{\lu})}$ is a source of $(\Q, \I)$ or there exists an integer $1\le y \le \ell_{\lu}+1$ such that
        $a_{\ell_{\lu}+1}a_{\ell_{\lu}}\cdots a_y$ is a generator of $\I$,
        where $a_{\ell_{\lu}+1}$ is an arrow ending at $v_{\lu}^{(\ell_{\lu})}$.
    \end{itemize}
Thus,  $\M^{-1}(D)=a_{\ell_{\lu}} \cdots a_2a_1 b_1^{-1}b_2^{-1}\cdots b_{\ell_{\ru}}^{-1}$ is an injective string,
then $M$ is an injective module.
\end{proof}
\checks{The formula (\ref{formula:0-cosyzygy-I}) also holds for $v_0$ being a vertex of the Type $(0^{\inner},2^{\out})$.
In this case, $D$ is both simple and injective.}
\checks{If $v_0$ is a vertex of the Type $(2^{\inner}, 1^{\out})$, then $P(v_0)$ is a string module corresponding to a directed string.
The following lemma provides a description of the $1$st-cosyzygy of any indecomposable module corresponding to a directed string.}

\begin{lemma} \label{lem:direct str. mod.}
Let $A=\kk\Q/\I$ be a string algebra and $M$ be an indecomposable $A$-module corresponding to a directed string $s$.
Then $\mho_1(M) \cong D_1 \oplus D_2$, where $D_1$ and $D_2$ are string modules corresponding to directed strings.
\end{lemma}

\begin{proof}
Since $s = a_1\cdots a_{\ell}$ is direct. Thus,
\begin{center}
$\M^{-1}(\EE(M)) = \wp' a_1\cdots a_{\ell} \wp''^{-1}$,
\end{center}
where $\wp' = a_1'\cdots a_m'$ ($m\ge 0$) is a path ending at $\s(a_1)$
and $\wp'' = a_1''\cdots a_n''$ ($n\ge 0$) is a path ending at $\t(a_{\ell})$. It should be noted that if $m=0$, then $\wp'$ is a path of length zero.
Then we obtain that $$\mho_1(M) = \M(a_1'\cdots a_{m-1}') \oplus \M(a_1''\cdots a_{n-1}'')$$
for all $m\ge 1 \text{ and } n\ge 1$. In particular, if $m = 0$ and $n = 0$, then $\mho_1(M)=0$.
Therefore, $\mho_1(M)$ is always a direct sum of two string modules corresponding to directed strings.
\end{proof}

\begin{remark}\rm \label{rmk:direct str. mod.}
%The dual of Lemma \ref{lem:direct str. mod.} holds, that is,
%$\Omega^1(M) \cong Q_1 \oplus Q_2$, where $Q_1$ and $Q_2$ are string modules corresponding to directed strings.
If $M$ is simple, then the string corresponding to $M$ is also a directed string with length zero.
Thus the dual of Lemma \ref{lem:direct str. mod.} can be used to compute the syzygies of any simple module.
\end{remark}

\subsection{Effective intersections}
In this subsection, we introduce the effective intersection of relations. In order to give the definition of effective intersections of relations, we introduce maximal paths with respect to a given path.
Let $p$ be a path in \checks{quiver $\Q$ which is} shown in \Pic \ref{fig:hat p}.
\begin{itemize}
  \item[(1)] If the quiver $\Q$ has no oriented cycle, then $p$ can be viewed as a subpath of the path $p_{\u}p$ where $p_{\u}$ is \checks{a} path \checks{in $\Q$} with no arrows ending at $v_1$;
\end{itemize}
\begin{figure}[htbp]
\centering
\begin{tikzpicture}[scale=1.75]
\draw (0,0) node{$v_1$};
\draw[->] (0.2,0) -- (0.8,0); \draw (0.5,0) node[above]{$a_1$};
\draw (1,0) node{$\cdots$};
\draw[->] (1.2,0) -- (1.8,0); \draw (1.5,0) node[above]{$a_{l-1}$};
\draw (2,0) node{$v_l$};
\draw[->] (2.2,0) -- (2.8,0); \draw (2.5,0) node[above]{$a_{l}$};
\draw (3,0) node{$v_{l+1}$};
\draw[->] (3.2,0) -- (3.8,0); \draw (3.5,0) node[above]{$a_{l+1}$};
\draw (4,0) node{$\cdots$};
\draw[->] (4.2,0) -- (4.8,0); \draw (4.5,0) node[above]{$a_{l+t-1}$};
\draw (5,0) node{$v_{l+t}$};
\draw[->] (5.2,0) -- (5.8,0); \draw (5.5,0) node[above]{$a_{l+t}$};
\draw (6.05,0) node{$v_{l\!+\!t\!+\!1}$};
%\draw[->] (6.2,0) -- (6.8,0); \draw (6.5,0) node[above]{$a_{l+t+1}$};
%\draw (7,0) node{$\cdots$};
%\draw[->] (7.2,0) -- (7.8,0); \draw (7.5,0) node[above]{$a_{n-1}$};
%\draw (8,0) node{$v_n$};
% p
\draw[red!50][line width =1pt] (4.5,0) ellipse (1.7cm and 0.3cm);
\draw[red] (4.5,-0.3) node[below]{$p=a_{l+1}\cdots a_{l+t}$};
% p left
\draw[blue!50][line width =1pt] (1.5,0) ellipse (1.7cm and 0.3cm);
\draw[blue] (1.5,-0.3) node[below]{$p_{\u}=a_1\cdots a_l$};
% p right
%\draw[green][line width =1pt][opacity=0.5] (7,0) ellipse (1.25cm and 0.3cm);
%\draw[green] (7,-0.3) node[below]{$p_{\d}=a_{l+t+1}\cdots a_{n}$};
%
\draw[opacity=0] (0, 0.5) node{$x$};
\end{tikzpicture}
\caption{A maximal path with respect to $p$}
\label{fig:hat p}
\end{figure}
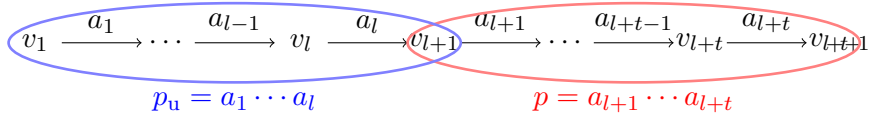
\begin{itemize}
  \item[(2)] If the quiver $\Q$ has oriented cycles, then $p$ may be a subpath of the path $p_{\u}p$ with infinite length \checks{that takes the following form}
\end{itemize}
\begin{center}
\begin{tikzpicture}[scale=1.75]
\draw[opacity=0] (0, 0.2) node{$x$};
\draw (0,0) node{$\cdots$};
\draw[->][decorate, decoration
      ={snake,amplitude=.4mm,segment length=2mm,post length=1mm}]
      (0.2,0) -- (0.8,0);
\draw (0.5,0) node[above]{$q_n$};
\draw (1,0) node{$v$};
\draw[->][decorate, decoration
      ={snake,amplitude=.4mm,segment length=2mm,post length=1mm}]
      (1.2,0) -- (1.8,0);
\draw (1.5,0) node[above]{$q_{n-1}$};
\draw (2,0) node{$\cdots$};
\draw[->][decorate, decoration
      ={snake,amplitude=.4mm,segment length=2mm,post length=1mm}]
      (2.2,0) -- (2.8,0);
\draw (2.5,0) node[above]{$q_2$};
\draw (3,0) node{$v$};
\draw[->][decorate, decoration
      ={snake,amplitude=.4mm,segment length=2mm,post length=1mm}]
      (3.2,0) -- (3.8,0);
\draw (3.5,0) node[above]{$q_1$};
\draw (4,0) node{$v$};
\draw[->][decorate, decoration
      ={snake,amplitude=.4mm,segment length=2mm,post length=1mm}]
      (4.2,0) -- (4.8,0);
%\draw (4.5,0) node[above]{$q$};
\draw (5.1,0) node{$\t(p)$.};
\end{tikzpicture}
\end{center}
\checks{It should be noted that $q_i$ and $q_j$ may not be equal for any $1\leq i,j$ and $i\neq j$.}

In both cases, we call $p_{\u}p$ a {\defines left maximal path} with respect to $p$. According to the definition of string algebras, there may be two arrows ending at a vertex $v_{k}$ for any $l+1\leq k\leq l+t+1$.
Thus, \checks{left maximal path} with respect to $p$ is not unique. In particular, if $p$ is a path of length zero, we call $p_{\u}p$ a left maximal path with respect to $\s(p)=\t(p)$. For simplicity, we abbreviate \checks{$p_{\u}p$ as $\widehat{p}$} and denote by $\lmpset(p)$ the set of all left maximal paths with respect to $p$. Dually, we can define {\defines right maximal paths} with respect to $p$.

\begin{remark} \rm
\checks{Notice that in (2) the vertex $v$ may not equal to $\s(p)$.}
\checks{Consider the string algebra $A=\kk\Q/\I$ given in Example \ref{exp:str alg}.}
\checks{A left maximal path respect to $p=z$ is}
\begin{center}
  \checks{$\cdots 1 \To{a} 2 \overbrace{\To{b} 3 \To{c} 4 \To{d}  5 \To{e} 6 \To{f} 7 \To{g} 8 \To{h} 1 \To{a}}^{q_1}
    2 \To{x} 9 \overbrace{\To{z}}^{p} 10$}
\end{center}
\checks{whose length is infinity. Here, $v=2\ne\s(p)=9$. }
\checks{Moreover, the above path is also a left maximal path respect to $p'=xz$.
In this case, we have $v=2=\s(p')$.}
\end{remark}

\begin{definition} \rm
Let $A=\kk\Q/\I$ be a string algebra and $p_2$, $p_1$ be two subpaths of some left maximal path $\widehat{p}$ with respect to a path $p$ in $\Q$.
If
\begin{itemize}
  \item[(1)] the vertex $\s(p_1)$ is between $\s(p_2)$ and $\t(p_2)$,
    \checks{where $\s(p_1)\notin\{\s(p_2),\t(p_2)\}$};
  \item[(2)] the vertex $\t(p_2)$ is between $\s(p_1)$ and $\t(p_1)$,
    \checks{where $\s(p_2)\notin\{\s(p_1),\t(p_1)\}$},
\end{itemize}
then we say $(p_2,p_1)$ is a pair with intersecting relations and call $p_2$ {\defines left-intersects} with $p_1$.
\end{definition}

A sequence of paths $\{p_i\}_{i=1}^{m} = (p_m, \ldots, p_1)$ is called a {\defines  left-intersecting sequence} beginning \checks{at} $p_m$ if $(p_{i+1},p_i)$ is a pair with \checks{left-intersecting} relations for any $1\le i\le m-1$. For convenience, we denote by LIS the left-intersecting sequence.

\begin{remark} \rm
Assume that $(p_m, \ldots, p_1)$ is a left-intersecting sequence, then we know that $(p_t, \ldots, p_1)$ is also a left-intersecting sequence for any $2\leq t\leq m$.
\end{remark}

\checks{Next, we will introduce effective left-intersecting sequences, which will be used to compute the self-injective dimension of a string algebra. First of all, we need the following definition.}
% We will use intersecting relations to compute the self-injective dimension of a string algebra.
% Notice that the self-injective dimension of a string algebra depends only on a part of the left-intersecting sequences.
% Next we will call these left-intersecting sequences effective. Before this, we need the following definition.

\begin{definition} \label{def:dis} \rm
Let $\pnotation = \cdots a_{-1}a_0 a_1\cdots$ be a left maximal path with respect to a path $p$ in $\Q$ and  $\wp = a_0a_1\cdots a_l$ be a subpath of $\pnotation$. \checks{Let $v_i := \s(a_i)$}. For any integer $t$,
\begin{itemize}
  \item [(1)] if $t \ge 0$, we define the {\defines left-distance} on the path ${\pnotation}$ from $v_{-t}$ to $\wp$ as $d_{\pnotation}(v_{-t}, \wp) = t$.
  \item [(2)] if $t\geq l+1$, then we define the {\defines right-distance} on the path ${\pnotation}$ from $\wp$ to $v_t$ as $d_{\pnotation}(\wp, v_t) = t-(l+1)$.
\end{itemize}

\end{definition}

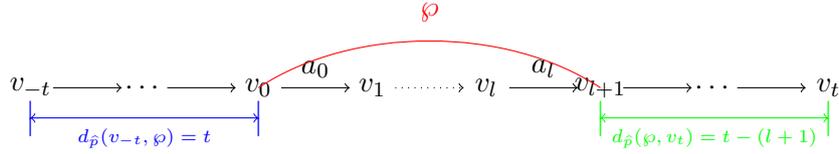
\begin{figure}[htbp]
\centering
\begin{tikzpicture}[xscale=1.5]
\draw (0,0) node{$v_0$} (1,0) node{$v_1$} (2,0) node{$v_l$} (3,0) node{$v_{l+1}$};
\draw[->] (0.2,0) -- (0.8,0); \draw (0.5,0) node[above]{$a_0$};
\draw[->][shift={(1,0)}] (0.2,0) -- (0.8,0) [dotted];
\draw[->][shift={(2,0)}] (0.2,0) -- (0.8,0); \draw[shift={(2,0)}] (0.5,0) node[above]{$a_l$};
\draw[red] (0,0) to[out=45,in=135] (3,0);
\draw[red] (1.5,1) node{$\wp$};
\draw (-2,0) node{$v_{-t}$} (-1,0) node{$\cdots$};
\draw[->][shift={(-2,0)}] (0.2,0) -- (0.8,0);
\draw[->][shift={(-1,0)}] (0.2,0) -- (0.8,0);
\draw[<->][blue] (-2,-0.4)node{$\big|$} -- (0,-0.4) node{$\big|$};
\draw[blue] (-1,-0.4) node[below]{\tiny $d_{\widehat{p}}(v_{-t},\wp)=t$};
\draw (5,0) node{$v_{t}$} (4,0) node{$\cdots$};
\draw[->][shift={(3,0)}] (0.2,0) -- (0.8,0);
\draw[->][shift={(4,0)}] (0.2,0) -- (0.8,0);
\draw[<->][green] (3,-0.4)node{$\big|$} -- (5,-0.4) node{$\big|$};
\draw[green] (4,-0.4) node[below]{\tiny $d_{\widehat{p}}(\wp,v_{t})=t-(l+1)$};
\end{tikzpicture}
\caption{Left and right-distances}
\label{fig:distance}
\end{figure}

Let $A=\kk\Q/\I$ be a string algebra and $s$ be a directed string. We denote by
\begin{align}
& \frakR_{\pnotation}^{\u}(s)
= \{r \in \I \text{ is a relation on } \pnotation \mid d_{\pnotation} (r, \t(s)) \ge 0 \}; \nonumber \\
& \frakR_{\pnotation}^{\d}(s)
= \{r \in \I \text{ is a relation on } \pnotation \mid d_{\pnotation} (\s(s), r) \ge 0 \}. \nonumber
\end{align}
In particular, if $\length(s)=0$, then we assume that $v = \s(s) = \t(s)$ and denote by
\begin{center}
$\frakR_{\pnotation}^{\u}(v) = \frakR_{\pnotation}^{\u}(s)$
and $\frakR_{\pnotation}^{\d}(v) = \frakR_{\pnotation}^{\d}(s)$.
\end{center}
Now we define the partial ordering relation $\prec$ as:
\begin{center}
$r_1 \prec r_2 \Leftrightarrow d_{\pnotation} (r_1, \t(s)) < d_{\pnotation} (r_2, \t(s)) (\mathrm{resp.}~d_{\pnotation} (\s(s), r_1) < d_{\pnotation} (\s(s), r_2)$).
\end{center}
It is clear that $(\frakR_{\pnotation}^{\u}(s), \prec)$ and  $(\frakR_{\pnotation}^{\d}(s), \prec)$ are posets. If $(\frakR_{\pnotation}^{\u}(s), \prec)$ is not empty, then it has \checks{a} unique minimal element. In this case, we denote by $r_{\pnotation}^{\u}(s)$ and $r_{\pnotation}^{\d}(s)$ the minimal elements of $(\frakR_{\pnotation}^{\u}(s), \prec)$ and  $(\frakR_{\pnotation}^{\d}(s), \prec)$, respectively.

\checks{Let $\nota$ be a left maximal path with respect to some path $p$ and $r_1 = r_{\nota}^{\u}(v) = a_{t_2}\cdots a_2a_1 $ ($0<t_1<t_2$) be a relation ending at $w_0$ (see \Pic \ref{fig:ELIS on rho}).
\begin{figure}[htbp]
\centering
\includegraphics[width=15cm]{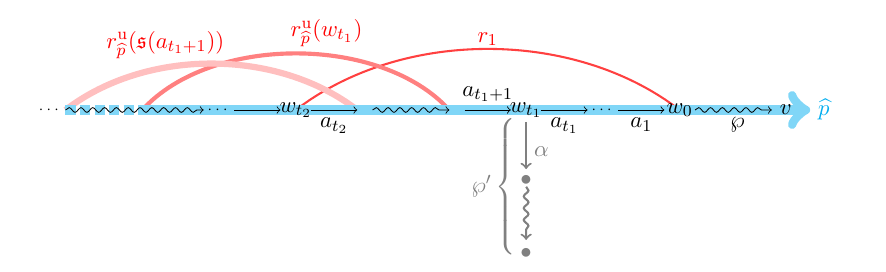} \\
\caption{If $\alpha$ exists, that is, $\length(\wp')>0$,
then $r_2=r_{\nota}^{\u}(w_{t_1})$ effectively left-intersects with $r_1$;
otherwise, $r_2=r_{\nota}^{\u}(\s(a_{t_1+1}))$
(it is possible for $r_{\nota}^{\u}(w_{t_1})=r_{\nota}^{\u}(\s(a_{t_1+1}))$)}
\label{fig:ELIS on rho}
\end{figure}
If there is a relation $r$ which left-intersecting with $r_1$ such that one of the following conditions $(a)$ and $(b)$ holds, then we take $r_2=r$.
\begin{itemize}
  \item[(a)]
    If $\length(\wp')>0$, then $r = r_{\nota}^{\u}(w_{t_1})$.
  \item[(b)]
    If $\length(\wp')=0$, then $r = r_{\nota}^{\u}(\s(a_{t_1+1}))$.
\end{itemize}
Furthermore,
\begin{itemize}
  \item[(c)] let $r'$ be a relation which left-intersecting with $r_2$ such that $r' = r_{\nota}^{\u}(\s(r_{1}))$. Denote $r$ by $r_3$, then we obtain a sequence $(r_3, r_2, r_1)$.
  \item[(d)] For any $i\geq 3$, repeat the above process $(c)$, if we obtain the relation $r_i$, then we can obtain a relation $r_{i+1}$ which left-intersecting with $r_i$ such that $r_{i+1} = r_{\nota}^{\u}(\s(r_{i-1}))$. As a consequence, we obtain a sequence $(r_n, \ldots, r_2, r_1)$ of relations on $\nota$.
\end{itemize}}

\begin{definition} \label{def:ELIS} \rm 
\begin{itemize}
\item[(1)] We call $(r_{i+1},r_i)$ a {\defines pair with effective left-intersecting relations}
($=$ELIR for short) of a string $s = \wp'^{-1}a_{t_1}\cdots a_1\wp$ for any $1\leq i\leq n-1$.

\item[(2)] We call the sequence $(r_n, \ldots, r_2, r_1)$ of relations on $\nota$ an {\defines effective left-intersecting sequence} ($=$ELIS for short) {\defines of $s$}.
\end{itemize}
\end{definition}

Note that if $(r_n,\ldots,r_1)$ is an ELIS of $s$, then so is $(r_t,\ldots, r_1)$ for all $1\le t\le n$. Moreover, for any $1\leq i\leq n-1$, $(r_{i+1},r_i)$ is also an ELIR of the string $s'$ corresponding to the module $\mho_{i+1}(\M(s))$.

\begin{example} \label{exp:ELIS} \rm
%% ========== 2025-2-22 00:46:37 ==========
%Now we compute the ELISs of $P(2)$ by Definition \ref{def:ELIS}.
%Since the vertex $5$ corresponding to $S(5)$ is a vertex on $\widehat{bcd}$, by the Step 1 of Definition \ref{def:ELIS}, we have $r_1=abcd$. Next because $\M^{-1}(P(2)) = z^{-1}x^{-1}bcd$, by Step 2 of Definition \ref{def:ELIS}, we obtain that $r_2=r_{\widehat{bcd}}^{\u}(2)$ left-intersecting with $r_1$. It is clear that $r_2=fgha$. %By Step 3, we have $r_3=r_{\widehat{bcd}}^{\u}(\s(r_1))$ left-intersecting with $r_2$.
%According to the quiver $\Q$, $r_3=efgh$. Repeat Step 3, we obtain an ELIS $(\ldots,r_5,r_4,r_3,r_2,r_1)$ of $P(2)$ on $\widehat{bcd}$, see \Pic \ref{fig:exp1}. Similarly, $(\ldots,r_5,r_4,r_3,r_2,r_1')$ is an ELIS of $P(2)$  on $\widehat{xz}$, see \Pic \ref{fig:exp2}.
%Moreover, By Remark \ref{rmk:ELIS}, we have
%\begin{align} \label{formula:ELIS of P(2)}
%  \ELIS(z^{-1}x^{-1}bcd) = \ \{
%& (r_1), (r_2, r_1), (r_3, r_2, r_1),\cdots; \nonumber\\
%& (r_1'), (r_2, r_1'), (r_3, r_2, r_1'), \cdots\}
%\end{align}
%% ========================================
\checks{
Let $A=\kk\Q/\I$ be the string algebra given in Example \ref{exp:str alg}.
Consider the projective string $z^{-1}x^{-1}bcd$ whose source is $2$.
For the path $\widehat{bcd} = \cdots (bcdefgha) (bcdefgha) bcd$, we have
\[ r_1 = r_{\widehat{bcd}}^{\u}(5) = abcd. \]
Since $2$ is a vertex of the Type $(1^{\inner},2^{\out})$,
by process $(a)$, we have
\[r_2 = r_{\widehat{bcd}}^{\u}(2) = fgha. \]
Furthermore, by the process $(c)$, we can obtain
\[r_3 = r_{\widehat{bcd}}^{\u}(\s(r_1)) =  r_{\widehat{bcd}}^{\u}(1) = efgh. \]
Repeating the process $(d)$, we obtain an ELIS $(\ldots, r_5,r_4,r_3, r_2, r_1)$ of $P(2)$ on $\widehat{bcd}$ which is shown in \Pic \ref{fig:exp1}.}
\begin{figure}[htbp]
\centering
\includegraphics[width=16cm]{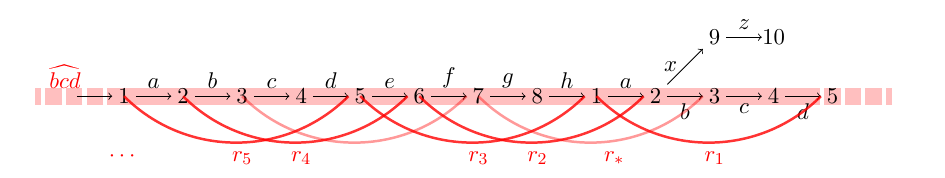}\\
\caption{The ELIS of $P(2)$ on $\widehat{bcd}$}
\label{fig:exp1}
\end{figure}
\checks{In this figure, one can see that $(r_*, r_1)$ is also a pair with left-intersecting relations,
but it is not effective. }
Similarly, see \Pic \ref{fig:exp2}, $(\ldots,r_5,r_4,r_3,r_2,r_1')$ is an ELIS of $P(2)$  on $\widehat{xz}$.
\begin{figure}[htbp]
\centering
\includegraphics[width=16cm]{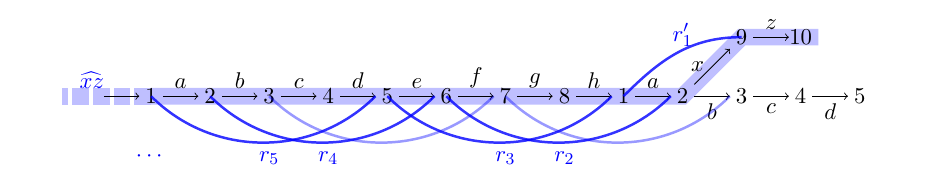}\\
\caption{The ELIS of $P(2)$ on $\widehat{xz}$}
\label{fig:exp2}
\end{figure}
\end{example}

Keep the notations from Proposition \ref{prop:0-cosyzygy-I} and denote by $s_a = a_{\ell_{\lu}}\cdots a_1$, $s_b = b_{\ell_{\ru}}\cdots b_1$,
$s_c = c_1 \cdots c_{\ell_{\ld}}$ and $s_d = d_1 \cdots d_{\ell_{\rd}}$.

\begin{itemize}
  \item[\checks{(1)}] If the minimal element of $\frakR_{\widehat{s_as_d}}^{\u}(v_{\rd}^{(\ell_{\rd})})$ is a relation with subpath $a_1d_1$, then we have
  $S(v_{\lu}^{(\ell_{\lu})})
    \le_{\oplus} \top D \le_{\oplus} \top(\EE(P(v_0))) = \top(\mho_1(P(v_0)))$.

  \item[\checks{(2)}] If the minimal element of $\frakR_{\widehat{s_bs_c}}^{\u}(v_{\ld}^{(\ell_{\ld})})$ is a relation with subpath
$b_1c_1$, then we have
  $S(v_{\ru}^{(\ell_{\ru})})
    \le_{\oplus} \top D \le_{\oplus} \top(\EE(P(v_0))) = \top(\mho_1(P(v_0)))$.
\end{itemize}

\checks{Let $v_0$ be a vertex of the Type $(2^{\inner}, 2^{\out})$, and define}
\begin{center}
$ \frakR_{\Left} = \big\{r \in \bigcup\limits_{\widehat{p}\in\lmpset(s_as_d)} \frakR_{\widehat{p}}^{\u} (
    v_0
    ) \mid (r,r_{\widehat{s_as_d}}^{\u}(v_{\rd}^{(\ell_{\rd})})) \text{  is a pair with intersecting relations} \big\}$;
\end{center}
and
\begin{center}
$ \frakR_{\Right} = \big\{ r \in \bigcup\limits_{\widehat{p}\in\lmpset(s_bs_c)} \frakR_{\widehat{p}}^{\u} (
    v_0
    ) \mid  (r ,r_{\widehat{s_bs_c}}^{\u}(v_{\ld}^{(\ell_{\ld})})) \text{  is a pair with intersecting relations} \big\}$.
\end{center}
Then we have the following proposition.

\begin{proposition} \label{prop:distance-proj case}
$\mho_1(D)$ is a direct sum of at most two string modules corresponding to directed strings. Moreover, $\mho_1(D)$ is injective if and only if
\begin{align}
\frakR_{\Left}
= \varnothing
\ \text{and}\ \
\frakR_{\Right}
= \varnothing. \label{formula:distance-proj case}
\end{align}
\end{proposition}

\begin{proof}
\checks{See \Pic \ref{fig:distance-proj case}, the cosyzygy $\mho_1(D)$ is a direct sum of two indecomposable modules $\M(s')$ and $\M(s'')$,
where $s'$ is a directed string on $\widehat{s_as_d}$,
and $s''$ is a directed string on $\widehat{s_bs_c}$}.

\begin{figure}[htbp]
\centering
\includegraphics[width=16cm]{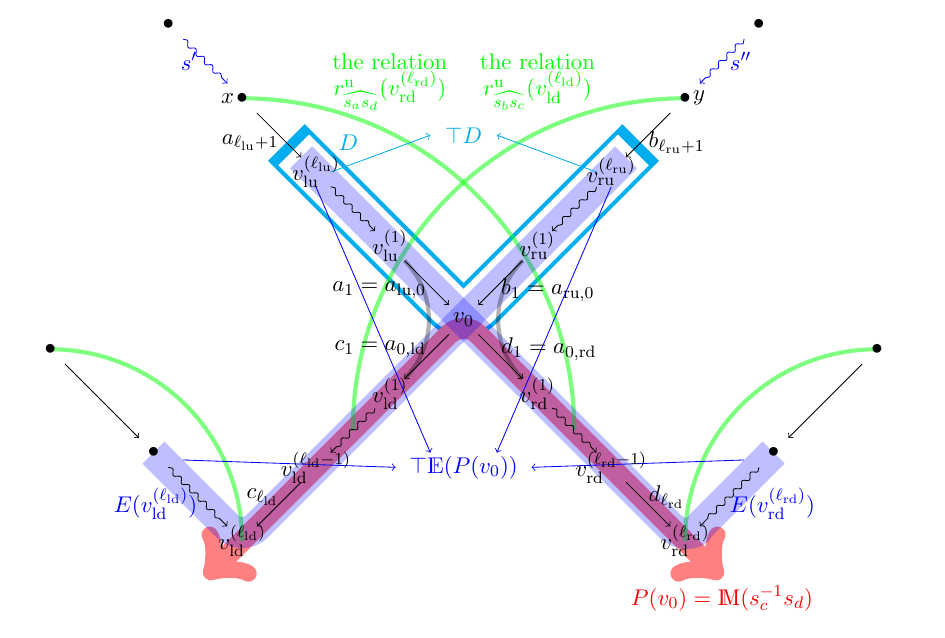}\\
\caption{$\EE(P(v_0)) \cong E(v_{\ld}^{(\ell_{\ld})}) \oplus E(v_{\rd}^{(\ell_{\rd})})$ }
\label{fig:distance-proj case}
\end{figure}

\checks{Let $\frakR_{\Left} = \varnothing$. Then we have the following two facts. }
\begin{itemize}
  \item[(1)] \checks{
    For each subpath $\wp$ of $\widehat{s_as_d}$ which is of the form $a_{\ell_{\lu}+m}\cdots a_{\ell_{\lu}+2}$ ($m\ge 2$),
    we have $\wp a_{\ell_{\lu}+1}s_a \notin \I$.
    Otherwise, $\wp a_{\ell_{\lu}+1}s_a$ contains a subpath which is a relation in $\frakR_{\widehat{s_as_d}}^{\u}(v_{\lu}^{(\ell_{\lu})})$ $(\subseteq \frakR_{\Left})$, this is a contradiction.
    }

  \item[(2)]
    The vertex \checks{$x$} is of the Type $((\le 1)^{\inner}, (\le 2)^{\out})$.
    Indeed, if there are two arrows $\alpha_1$ and $\alpha_2$ with $\t(\alpha_1) = \checks{x} = \t(\alpha_2)$,
    then, by the definition of string algebra, at least one of $\alpha_1 a_{\ell_{\lu+1}}$
    and $\alpha_2 a_{\ell_{\lu+1}}$ belong to $\I$. For any $j\in \{1,2\}$, if $\alpha_j a_{\ell_{\lu+1}} \in \I$, then $\alpha_j a_{\ell_{\lu+1}} \in \frakR_{\widehat{s_as_d}}^{\u}(v_{\lu}^{(\ell_{\lu})}) \checks{(\subseteq \frakR_{\Left}})$, this is a contradiction.
\end{itemize}
\checks{Then $s'$ is a directed string on $\widehat{s_as_d}$ whose ending point is $x$,
and it is not a subpath of any relation in $\frakR_{\Left}$.
We obtain that $s'$ is left maximal, that is, $\M(s')$ is isomorphic to the indecomposable injective module $E(x)$. }

\checks{Let} $\frakR_{\Left} \ne \varnothing$. Then there is $1\le i\le \ell_{\lu}$ such that $r_{\widehat{s_as_d}}^{\u}(v_0) =
pa_{\ell_{\lu}}\cdots a_i$, where $p$ is a path on $\widehat{s_as_d}$ which is \checks{of the following form:}
\[ \checks{ p = a_{\ell_{\lu}+n}\cdots a_{\ell_{\lu}+2}a_{\ell_{\lu}+1}\ (n\ge 2). } \]
In this case, the socle of $\M(s')$ is \checks{$S(x)$},
\checks{and we have $s' = a_{\ell_{\lu}+(n-1)}\cdots a_{\ell_{\lu}+2}$
(note that if $n=2$, then $s'$ is the path of length zero corresponding to $x$).
Since $a_{\ell_{\lu}+n}s'$ does not belong to $\I$,
we have that $\M(s')$ is non-injective.}
\checks{Therefore, $\M(s')$ is injective if and only if $\frakR_{\Left}=\varnothing$.}

\checks{We can consider the injectivity of $\M(s'')$ in a similar way, and then this proposition holds. }
% ========== 2025-2-22 11:11:01 ==========
%It follows that $0 \ne \M(s') \le_{\oplus} \mho_1(D)$.
%Then $s'$ is a subpath of $p$ whose starting point is $\s(p)$.
%Thus, we obtain $\mho_1(X_1)\ne 0$, this shows that $\mho_1(D)$ is non-injective.
%
%Thus, there exists at most one arrow $\alpha$ such that $\t(\alpha) = v_{\lu}^{(\ell_{\lu+1})}$ and $\alpha a_{\lu+1} \notin \I$.
%We obtain that \[\mho_1(D) \cong \M(s') \oplus \M(s'') \cong E(x) \oplus E(y)\]
%is injective, where $x$ and $y$ are points shown \Pic \ref{fig:distance-proj case}; $s'$ and $s''$ are strings whose endpoints are $x$ and $y$, respectively.
% ========================================
\end{proof}

\begin{remark} \label{rmk:ELIS} \rm
\begin{itemize}
  \item[(1)] The effective intersection of two relations is introduced by Yang and Zhang in \cite{YZhang2019}. They computed the global dimension of string algebras which are of types $\A_n$ and $\tA_n$ with linear orientations.
  \item[(2)] Notice that any indecomposable module over a string algebra can be described by string, thus any ELIS with respect to a directed string $s$ can also be called an ELIS of $\M(s)$.
\end{itemize}
\end{remark}

We directly obtain the following corollary by Proposition \ref{prop:distance-proj case}.

\begin{corollary} \label{coro}
Let $\checks{p}=a_{l+1}\cdots a_{l+t}$ be a string as \Pic \ref{fig:hat p} and $\widehat{p}$ be a left maximal path in $\lmpset(\t(\checks{p}))$.
\begin{itemize}
  \item[\rm(1)] If $\frakR_{\widehat{p}}^{\u}(\checks{\t(p)}) = \varnothing$,
    then $\mho_1(\M(\checks{p}))$ contains an injective direct summand $\M(s')$,
    where $s'$ is a directed string on $\widehat{p}$.
    Furthermore, if there is another left maximal path $\widehat{q} \checks{ \in \lmpset(\t(p))}$,
    \checks{where $q=\alpha_1\alpha_2\cdots \alpha_m$ is a path ending at $\t(p)$ such that $\alpha_m\ne a_{l+t}$},
    then $\mho_1(\M(\checks{p})) \cong \M(s') \oplus \M(s'')$, where $s''$ is a directed string on $\widehat{q}$.
    In particular, if $\frakR_{\widehat{q}}^{\u}(\checks{\t(p)}) \ \checks{(=\frakR_{\widehat{q}}^{\u}(\checks{\t(q)})}\ = \varnothing$,
    then $\M(s'')$ is injective.
  \item[\rm(2)] If $\frakR_{\widehat{p}}^{\u}(\checks{\t(p)}) \ne \varnothing$ and $\mho_1(\M(p))$ is non-injective, then $\t(p)$ is a vertex of Type $(2^{\inner}, (\ge 0)^{\out})$ or \checks{$r_{\widehat{p}}^{\u}(\s(a_{l+t}))$} left-intersecting with \checks{$r_{\widehat{p}}^{\u}(\t(p))$}.
\end{itemize}
\end{corollary}

\begin{remark} \rm \label{rem:a and b}
By Proposition \ref{prop:0-cosyzygy-I} and Corollary \ref{coro}, we have obtained two kinds intersection relations.
\begin{itemize}
  \item[(I)]
    \checks{
    $r_2 = r_{\widehat{s_as_d}}^{\u}(v_0)$ left-intersecting with $r_1 = r_{\widehat{s_as_d}}^{\u}(v_{\rd}^{(\ell_{\rd})})$ such that
    $(r_2,r_1)$ is an ELIS of $P(v_0) \cong \M(s_c^{-1}s_d)$.
    Here, $v_0$ is the source of the string corresponding to $P(v_0)$, see \Pic \ref{fig:int a}. Indeed, this type corresponds to (a).
    }
    % $c_{\ell_{\ld}}^{-1}\cdots c_2^{-1}c_1^{-1} d_1d_2\cdots d_{\ell_{\rd}}$
%
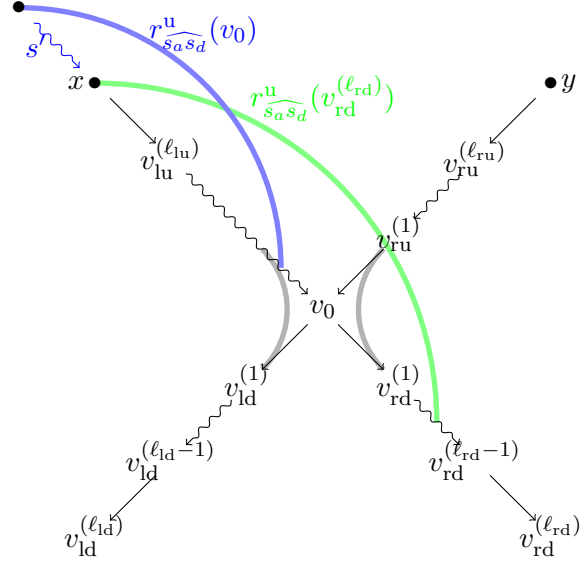
\begin{figure}[htbp]
\begin{tikzpicture}
% relations
% lu
\draw[green!50][line width=2pt] ( 1.5,-1.5) to[out= 90,in=   0] (-3, 3);
\draw[green] (-0,2.8) node{$r_{\widehat{s_as_d}}^{\u}(v_{\rd}^{(\ell_{\rd})})$};
\draw[blue!50][line width=2pt] (-0.55, 0.55) to[out= 90,in=  0] (-4.0, 4.0);
\draw[blue] (-1.57,3.65) node{$r_{\widehat{s_as_d}}^{\u}(v_0)$};
%ru
%\draw[green][opacity=0.5][line width=2pt] (-1.5,-1.5) to[out=90,in=180] ( 3, 3);
%\draw[green] (0.92,3.0) node{$r_{\widehat{s_bs_c}}^{\u}(v_{\ld}^{(\ell_{\ld})})$};
%\draw[blue][opacity=0.5][line width=2pt] ( 1.5, 1.5) to[out= 90,in= 180] ( 4, 4);
%\draw[blue] (1.77,3.65) node{$r_{\widehat{s_bs_c}}^{\u}(v_0)$};
% ========== graph ==========
\draw (0,0) node{$v_0$};
%\draw (-1, 1) node{$v_{\lu}^{(1)}$};
\draw ( 1, 1) node{$v_{\ru}^{(1)}$};
\draw ( 1,-1) node{$v_{\rd}^{(1)}$};
\draw (-1,-1) node{$v_{\ld}^{(1)}$};
% relations on v0
\draw [black!30][line width = 2pt] (-0.8, 0.8) to[out= -45,in= 45] (-0.8,-0.8);
\draw [black!30][line width = 2pt] ( 0.8, 0.8) to[out=-135,in=135] ( 0.8,-0.8);
% arrows
% lu
%\draw[->] (-0.8, 0.8) -- (-0.2, 0.2);
%\draw (-0.4, 0.4) node[left]{$a_1$};
%\draw (-2.4, 2.4) node[left]{$a_{\ell_{\lu}+1}$};
\draw[->][decorate, decoration
      ={snake,amplitude=.4mm,segment length=2mm,post length=1mm}] (-1.8, 1.8) -- (-0.2, 0.2);
\draw (-2, 2) node{$v_{\lu}^{(\ell_{\lu})}$};
\draw[->] (-2.8, 2.8) -- (-2.2, 2.2);
\draw (-3.0, 3.0) node{$\bullet$};
\draw (-3.0, 3.0) node[left]{$x$};
\draw[->][blue, decorate, decoration
      ={snake,amplitude=.4mm,segment length=2mm,post length=1mm}] (-3.8, 3.8) -- (-3.2, 3.2);
\draw[blue] (-3.5,3.5) node[left]{$s'$};
\draw (-4.0, 4.0) node{$\bullet$};
% ru
\draw[->][decorate, decoration
      ={snake,amplitude=.4mm,segment length=2mm,post length=1mm}] ( 1.8, 1.8) -- ( 1.2, 1.2);
\draw ( 2, 2) node{$v_{\ru}^{(\ell_{\ru})}$};
\draw[->] ( 2.8, 2.8) -- ( 2.2, 2.2);
\draw ( 3.0, 3.0) node{$\bullet$};
\draw ( 3.0, 3.0) node[right]{$y$};
%\draw[->][blue, decorate, decoration
%      ={snake,amplitude=.4mm,segment length=2mm,post length=1mm}] ( 3.8, 3.8) -- ( 3.2, 3.2);
%\draw[blue] ( 3.5,3.5) node[right]{$s''$};
%\draw ( 4.0, 4.0) node{$\bullet$};
\draw[->] ( 0.8, 0.8) -- ( 0.2, 0.2);
%\draw ( 0.4, 0.4) node[right]{$b_1$};
%\draw ( 2.4, 2.4) node[right]{$b_{\ell_{\ru}+1}$};
% ld
\draw[->] (-0.2,-0.2) -- (-0.8,-0.8);
%\draw (-0.4,-0.4) node[left]{$c_1$};
\draw[->][decorate, decoration
      ={snake,amplitude=.4mm,segment length=2mm,post length=1mm}] (-1.2,-1.2) -- (-1.8,-1.8);
\draw (-2,-2) node{$v_{\ld}^{(\ell_{\ld}-1)}$};
\draw[->] (-2.2,-2.2) -- (-2.8,-2.8);
\draw (-3.0,-3.0) node{$v_{\ld}^{(\ell_{\ld})}$};
%\draw (-2.4,-2.4) node[left]{$c_{\ell_{\ld}}$};
% rd
\draw[->] ( 0.2,-0.2) -- ( 0.8,-0.8);
%\draw ( 0.4,-0.4) node[right]{$d_1$};
\draw[->][decorate, decoration
      ={snake,amplitude=.4mm,segment length=2mm,post length=1mm}] ( 1.2,-1.2) -- ( 1.8,-1.8);
\draw ( 2,-2) node{$v_{\rd}^{(\ell_{\rd}-1)}$};
\draw[->] ( 2.2,-2.2) -- ( 2.8,-2.8);
\draw ( 3.0,-3.0) node{$v_{\rd}^{(\ell_{\rd})}$};
%\draw ( 2.4,-2.4) node[right]{$d_{\ell_{\rd}}$};
\end{tikzpicture}
\caption{An ELIS given by $\M(s_c^{-1}s_d)$ ($\ell_{\ld}\ge 1$, $\ell_{\rd}\ge 1$)}
\label{fig:int a}
\end{figure}

  \item[(II)]
  \checks{
  In the case of $\length(s_c) = 0$ (the case of $\length(s_d)=0$ is symmetric),
  $r_2=r_{\widehat{s_as_d}}^{\u}(v_{\lu}^{(1)})$ left-intersecting with $r_1=r_{\widehat{s_as_d}}^{\u}(v_{\rd}^{(\ell_{\rd})})$ such that $(r_2,r_1)$ is an ELIS of $P(v_0) \cong \M(s_d)$.
  Here, $v_{\lu}^{(1)}$ is not the source of the string corresponding to $P(v_0)$, see \Pic \ref{fig:int b}. Indeed, this type corresponds to (b).
  }
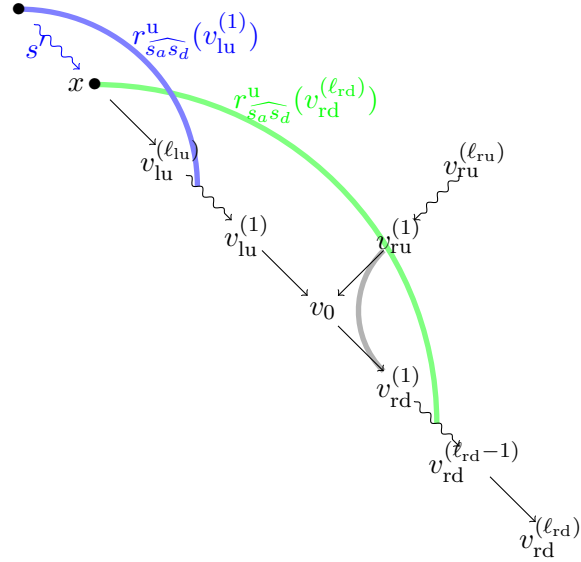
\begin{figure}[htbp]
\begin{tikzpicture}
% relations
% lu
\draw[green!50][line width=2pt] ( 1.5,-1.5) to[out= 90,in=   0] (-3, 3);
\draw[green] (-0.22,2.8) node{$r_{\widehat{s_as_d}}^{\u}(v_{\rd}^{(\ell_{\rd})})$};
\draw[blue!50][line width=2pt] (-1.65, 1.65) to[out= 90,in=  0] (-4.0, 4.0);
\draw[blue] (-1.67,3.65) node{$r_{\widehat{s_as_d}}^{\u}(v_{\lu}^{(1)})$};
%ru
%\draw[green][opacity=0.5][line width=2pt] (-1.5,-1.5) to[out=90,in=180] ( 3, 3);
%\draw[green] (0.92,3.0) node{$r_{\widehat{s_bs_c}}^{\u}(v_{\ld}^{(\ell_{\ld})})$};
%\draw[blue][opacity=0.5][line width=2pt] ( 1.5, 1.5) to[out= 90,in= 180] ( 4, 4);
%\draw[blue] (1.77,3.65) node{$r_{\widehat{s_bs_c}}^{\u}(v_0)$};
% ========== graph ==========
\draw (0,0) node{$v_0$};
\draw (-1, 1) node{$v_{\lu}^{(1)}$};
\draw ( 1, 1) node{$v_{\ru}^{(1)}$};
\draw ( 1,-1) node{$v_{\rd}^{(1)}$};
% relations on v0
\draw [black!30][line width = 2pt] ( 0.8, 0.8) to[out=-135,in=135] ( 0.8,-0.8);
% arrows
% lu
\draw[->] (-0.8, 0.8) -- (-0.2, 0.2);
%\draw (-0.4, 0.4) node[left]{$a_1$};
%\draw (-2.4, 2.4) node[left]{$a_{\ell_{\lu}+1}$};
\draw[->][decorate, decoration
      ={snake,amplitude=.4mm,segment length=2mm,post length=1mm}] (-1.8, 1.8) -- (-1.2, 1.2);
\draw (-2, 2) node{$v_{\lu}^{(\ell_{\lu})}$};
\draw[->] (-2.8, 2.8) -- (-2.2, 2.2);
\draw (-3.0, 3.0) node{$\bullet$};
\draw (-3.0, 3.0) node[left]{$x$};
\draw[->][blue, decorate, decoration
      ={snake,amplitude=.4mm,segment length=2mm,post length=1mm}] (-3.8, 3.8) -- (-3.2, 3.2);
\draw[blue] (-3.5,3.5) node[left]{$s'$};
\draw (-4.0, 4.0) node{$\bullet$};
% ru
\draw[->][decorate, decoration
      ={snake,amplitude=.4mm,segment length=2mm,post length=1mm}] ( 1.8, 1.8) -- ( 1.2, 1.2);
\draw ( 2, 2) node{$v_{\ru}^{(\ell_{\ru})}$};
%\draw[->][blue, decorate, decoration
%      ={snake,amplitude=.4mm,segment length=2mm,post length=1mm}] ( 3.8, 3.8) -- ( 3.2, 3.2);
%\draw[blue] ( 3.5,3.5) node[right]{$s''$};
%\draw ( 4.0, 4.0) node{$\bullet$};
\draw[->] ( 0.8, 0.8) -- ( 0.2, 0.2);
%\draw ( 0.4, 0.4) node[right]{$b_1$};
%\draw ( 2.4, 2.4) node[right]{$b_{\ell_{\ru}+1}$};
%\draw (-2.4,-2.4) node[left]{$c_{\ell_{\ld}}$};
% rd
\draw[->] ( 0.2,-0.2) -- ( 0.8,-0.8);
%\draw ( 0.4,-0.4) node[right]{$d_1$};
\draw[->][decorate, decoration
      ={snake,amplitude=.4mm,segment length=2mm,post length=1mm}] ( 1.2,-1.2) -- ( 1.8,-1.8);
\draw ( 2,-2) node{$v_{\rd}^{(\ell_{\rd}-1)}$};
\draw[->] ( 2.2,-2.2) -- ( 2.8,-2.8);
\draw ( 3.0,-3.0) node{$v_{\rd}^{(\ell_{\rd})}$};
%\draw ( 2.4,-2.4) node[right]{$d_{\ell_{\rd}}$};
\end{tikzpicture}
\caption{An ELIS given by $\M(s_d)$ ($\ell_{\rd}\ge 1$)}
\label{fig:int b}
\end{figure}
\end{itemize}
\end{remark}

\section{Proof of the main results}
In this section, we give a proof of the main results. Let $A=\kk\Q/\I$ be a string algebra. Assume that $\{r_i\}_{i=1}^{n}=\{r_n,\ldots,r_1\}$ is an ELIS of projective module $P(v_0)$ for vertex $v_0$ in $\Q$. We define the length of $\ELIS(P(v_0))$ as:
\[\length(\ELIS(P(v_0))) :=
\begin{cases}
\sup\limits_{\{r_i\}_{i=1}^{n} \in \ELIS(P(v_0))} \length(\{r_i\}_{i=1}^{n}), & \text{ if $\ELIS(P(v_0)) \ne \varnothing$; } \\
\quad\quad\quad\quad\quad 0, &\text{ if $\ELIS(P(v_0)) = \varnothing$. }
\end{cases}\]

\begin{lemma} \label{lemm:ELIS and cosyzygy}
If $P(v_0)$ has an $\mathrm{ELIS}$ $\{r_i\}_{i=1}^{n} = (r_n, \ldots, r_1)$, then $\mho_{n+1}(P(v_0)) \ne 0$.
\end{lemma}

\begin{proof}
Keep the notations from \checks{\Pic \ref{fig:distance-proj case}},
\checks{the proof is divided into two cases as follows:}
\begin{itemize}
  \item[(1)] \checks{$v_0$ is a vertex of the Type $((\le 2)^{\inner}, 2^{\out})$;}
  \item[(2)] \checks{$v_0$ is a vertex of the Type $((\le 2)^{\inner}, (\le 1)^{\out})$.}
\end{itemize}

\checks{In (1), $\soc P(v_0)$ is a direct sum of two simple modules $S(\t(s_c))$ and $S(\t(s_d))$.}
\checks{We assume that $v_0$ is a vertex of the Type $(2^{\inner},2^{\out})$}
\checks{(the cases of $(1^{\inner}, 2^{\out})$ and $(0^{\inner}, 2^{\out})$ is similar).}
\checks{Then $\ell_{\lu}, \ell_{\ru}, \ell_{\ld}, \ell_{\rd} \ge 1$}.
Let $s_a = a_{\ell_{\lu}}\cdots a_1$, $s_b = b_{\ell_{\ru}}\cdots b_1$,
$s_c = c_1\cdots c_{\ell_{\ld}}$ and $s_d = d_1\cdots d_{\ell_{\rd}}$.
Then $\M^{-1}(P(v_0)) = s_c^{-1}s_d$. By Proposition \ref{prop:0-cosyzygy-I}, $\mho_1(P(v_0)) \cong D_{\Left}\oplus D\oplus D_{\Right}$. Thus, we have the following two cases:
\begin{itemize}
    \item[(1.1)] all relations in $\{r_i\}_{i=1}^{n}$ are on some left maximal path in
     $\lmpset(s_as_d)$ or $\lmpset(s_bs_c)$.
    \item[(1.2)] all relations in $\{r_i\}_{i=1}^{n}$ are on some left maximal path in
     $\lmpset(\wp_{\Left})$ or $\lmpset(\wp_{\Right})$.
\end{itemize}

Next, we prove subcase (1.1) by induction on the length of ELIS of $P(v_0)$.
Without loss of generality, \checks{let $\widehat{p} = \widehat{s_as_d} \in \lmpset(s_as_d)$, and}
assume that all $r_i$ are \checks{subpaths} of $\widehat{p}$. Then $r_1 = r^{\u}_{\widehat{p}}(v_{\rd}^{(\ell_{\rd})})$. By Proposition \ref{prop:0-cosyzygy-I}, $0\ne D \le_{\oplus} \mho_1(P(v_0))$.
Now assume that $P(v_0)$ has an ELIS $(r_2, r_1)$. Then $r_2 = r^{\u}_{\widehat{p}}(v_0)$ and $v_0$ is a vertex on $r_1$.
Since $v_0$ is not \checks{gently relational}, $D$ is not injective again by Proposition \ref{prop:0-cosyzygy-I}.
We obtain $0\ne \mho_1(D) \le_{\oplus} \mho_2(P(v_0))$.
By Lemma \ref{lem:direct str. mod.}, we have
\begin{center}
  $\mho_1(D) \cong \M(s_1') \oplus \M(s_1'')$,
\end{center}
where \checks{$s_1'$ and $s_1''$ are directed strings.}
\checks{If $P(v_0)$ has an ELIS $(r_3,r_2, r_1)$ on $\widehat{p}$,
that is, $n\ge 3$, then one of $s_1'$ and $s_1''$ can be seen as a subpath of $\widehat{p}$,
see \Pic \ref{fig:ELIS and cosyzygy}. }
\begin{figure}[htbp]
  \centering
  \includegraphics[width=16cm]{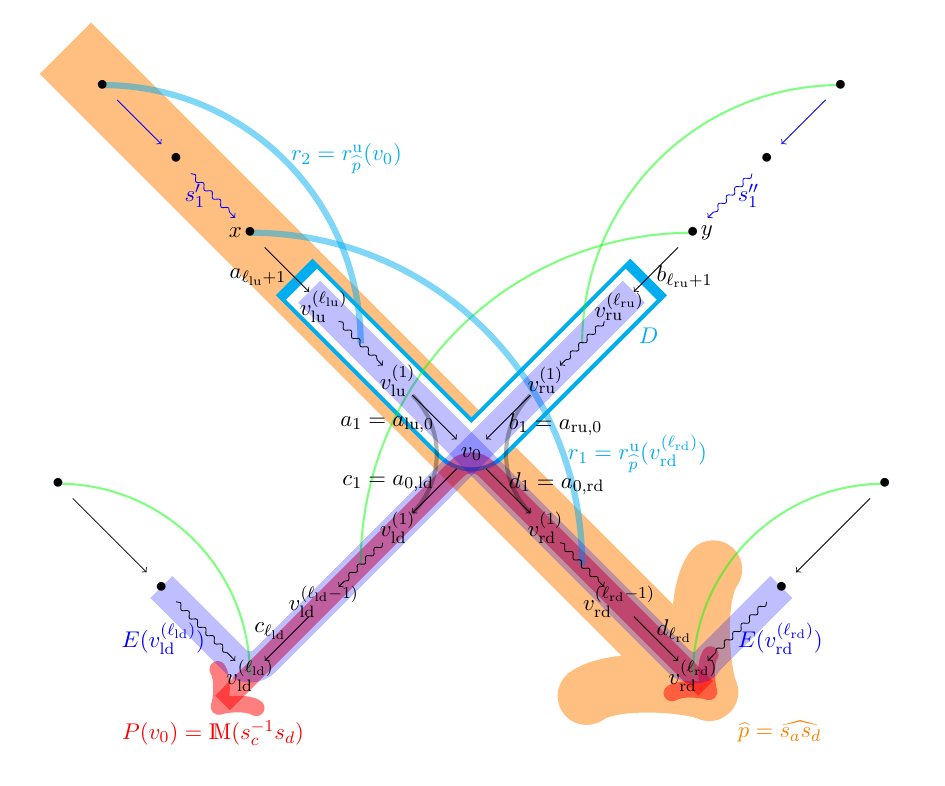}\\
\caption{$\EE(P(v_0)) \cong E(v_{\ld}^{(\ell_{\ld})}) \oplus E(v_{\rd}^{(\ell_{\rd})})$ }
\label{fig:ELIS and cosyzygy}
\end{figure}
\checks{Let $s_1'$ be a subpath of $\widehat{p}$, then we have $\soc \M(s_1') \cong S(\s(r_1)) = S(x) \ne 0$. Thus,}
\[ 0\ne \M(s_1')  \le_{\oplus} \mho_1(D) \le_{\oplus} \mho_2(P(v_0)).\]
The injective envelope of $\M(s_1')$ is
\[ e_0^{\M(s_1')} : \M(s_1') \to \EE(\soc\M(s_1')) \cong E(\s(r_1)), \]
where $E(\s(r_1))$ is decided by the type of the vertex $\s(r_1)$.
%(i.e, $\s(r_1)$ is one of the type $(1^{\inner}, (\ge 1)^{\out})$ and $(2^{\inner}, (\ge 1)^{\out})$)
%and the relations in $\frakR^{\u}_{\widehat{p}}(\s(r_1)) \cup \frakR^{\u}_{\widehat{q}}(\s(r_1))$
%which effective left-intersect to $r_2$ ($q$ is another path ending at $\s(r_1)$).
By Lemma \ref{lem:direct str. mod.},
$$\coker(e_0^{\M(s_1')}) \cong \M(s_2') \oplus \M(s_2''),$$
\checks{where $s_2'$ and $s_2''$ are directed strings, and one of $s_2'$ and $s_2''$ is a string on $\widehat{p}$.}

If $\frakR^{\u}_{\widehat{p}}(\s(r_1)) = \varnothing$ or any $r \in \frakR^{\u}_{\widehat{p}}(\s(r_1))$
is a relation such that $(r, r_2, r_1)$ is not an ELIS of $P(v_0)$, then \checks{we have $n=2$ in this case.}
By Proposition \ref{prop:distance-proj case}, $\M(s_2')$ is injective with socle $S(\s(r_2))$. Thus,
\begin{align}\label{formula:ELIS and cosyzygy}
0\ne \M(s_2') \le_{\oplus} \coker(e_0^{\M(s_1')})
\le_{\oplus} \mho_2(D) \le_{\oplus} \mho_3(P(v_0)).
\end{align}
Otherwise, $n\ge 3$, that is, we have that $(r_3,r_2,r_1)$ is an ELIS of $P(v_0)$.
Then $r_3 = r_{\widehat{p}}^{\u}(\s(r_1))$ and $\M(s_2')$ is non-injective.
Thus, (\ref{formula:ELIS and cosyzygy}) holds and further $\mho_3(P(v_0))\ne 0$.

%Finally, since $\mho_{\ge 2}(P(v_0))$ is a direct sum of some indecomposable modules corresponding to directed strings,
%we can repeat the above method. Thus, by induction, if $P(v_0)$ has an ELIS $\{r_i\}_{i=1}^{n} = (r_n, \ldots, r_1)$, then $\mho_{n+1}(P(v_0)) \ne 0$.

\checks{Repeating the steps as above, we can obtain an ELIS $(r_n,\ldots,r_2,r_1)$ of $P(v_0)$ and a non-zero direct summand of $\mho_{n+1}(P(v_0))$ by induction.}

\checks{In the subcase (1.2), we need to compute the cosyzygies of $D_{\Left}$ and $D_{\Right}$.
Here, $D_{\Left}$ and $D_{\Right}$ are modules corresponding to directed string.
In (2), the string of $P(v_0)$ is directed.
We can compute their cosyzygies by using the method similar to compute the cosyzygy of $\M(s_1')$,
and prove the other cases of this lemma.}
\end{proof}

Now we can give our main result.

\begin{theorem} \label{thm:main}
Let $A=\kk\Q/\I$ be a string algebra.
\begin{itemize}
  \item [\rm(1)] If $A$ is not self-injective, then
    \[ \id A = \sup_{v\in \Q_0} \length(\ELIS(P(v))) + 1. \]
  \item [\rm(2)] $A$ is Gorenstein if and only if all ELISs of $P(v)$ have finite length for any vertex $v$ in $\Q_0$,

\end{itemize}
\end{theorem}
\begin{proof}
For any vertex $v$ in $\Q$, we claim that
\[ \id P(v) = \begin{cases}
\length(\ELIS(P(v))) + 1, & \text{ if $P(v)$ is non-injective; }\\
\quad\quad\quad\quad\ 0, & \text{ if $P(v)$ is injective. }
\end{cases} \]
In this case, it is clear that if $A$ is not self-injective, then
\[ \id A = \sup_{v\in \Q_0} \length(\ELIS(P(v))) + 1. \] Moreover, $A$ is Gorenstein if and only if all ELISs of $P(v)$ have finite length. Now we prove the claim. If $\length(\ELIS(P(v))) = \infty$,
then there exists a sequence $(\ldots, r_2, r_1)$ of relations such that for any $N\ge 1$, $\{r_i\}_{N \ge i\ge 1}$ is an ELIS of $P(v)$.
By Lemma \ref{lemm:ELIS and cosyzygy}, $\mho_{N+1}(P(v)) \ne 0$.
Thus, we have $\id P(v) = \infty$.

If $\length(\ELIS(P(v))) = n < \infty$, let $\{r_i\}_{n \ge i\ge 1}$ be an ELIS with length $n$.
then $r_n$ is a relation in $\frakR_{\widehat{p}}^{\u}(v)$,
where $\widehat{p}$ is some left maximal path such that all relations in $\{r_i\}_{n \ge i\ge 1}$ are paths on the $\widehat{p}$.
Since there exists no relation $r_{n+1}$ such that $\{r_i\}_{n+1 \ge i\ge 1}$ is an ELIS of $P(v)$,
we have $\mho_{n+1}(P(v))$ is a non-zero injective module
by Propositions \ref{prop:distance-proj case}.
\end{proof}

\begin{remark} \rm \label{rek:type b}
Let $A=\kk\Q/\I$ be a gentle algebra \cite{ASS2006} and $(r_n,\ldots,r_1)$ be an ELIS with respect to indecomposable projective $A$-module $P(v)$.
Then \checks{all $r_i$ are relations of length two, and}
$r_{i+1}$ effective left-intersecting with $r_i$ are always intersection relations of type (II) in Remark \ref{rem:a and b} for any $\checks{2} \leq i\leq n-1$.
\end{remark}

Gei\ss ~and Reiten proved that gentle algebras are Gorenstein in \cite{GR2005}. Recently, Amiot, Plamondon, and Schroll \cite{APS2023} gave a characterization of the silting objects of the derived category of gentle algebras in terms of graded curves. According to this characterization, they \checks{gave new proof that gentle algebras are Gorenstein}. By using the effective left-intersecting sequence, we can also
obtain the following corollary.

\begin{corollary}[\rm \!\cite{GR2005}]
Gentle algebras are Gorenstein.
\end{corollary}

\begin{proof}
Let $A=\kk\Q/\I$ be a gentle algebra and $(\alpha_{n+1}\alpha_{n},\ldots, \alpha_2\alpha_1)$ be an ELIS with respect to indecomposable projective $A$-module $P(v)$.
Next, we prove that all relations are not on any full relational oriented cycle, where a full relational oriented cycle is an oriented cycle $C=a_0a_1\cdots a_{l-1}$ such that $a_{\overline{i}}a_{\overline{i+1}}\in\I$ ($\overline{i}$ is $i$ modulo $l$).

Indeed, if $\alpha_{i+1}\alpha_i$ is a relation on $C$,
then $\alpha_{i-1}$ is an arrow on $C$. Otherwise, by the definition of gentle algebra, we have $\alpha_i\alpha_{i-1}\notin\I$.
Thus, by induction, we obtain that $\alpha_{n+1}\alpha_n$, $\ldots$, $\alpha_2\alpha_1$ are relations on $C$
and $v = \checks{\t(\alpha_2) = \s(\alpha_1)}$ is a vertex on $C$. In other words, $\M^{-1}(P(v)) = \checks{p^{-1}\alpha_1}$,
where $p$ is a path starting at $v$ such that $\checks{\alpha_2p}\notin\I$.
We have the following two cases:
\begin{itemize}
  \item[(1)] If the length of $p$ is greater than or equal to $1$, then $\alpha_3\alpha_2$ does not \checks{satisfy} the condition given in processes (a) and (b).
  \item[(2)] If the length of $p$ is zero, then $\M^{-1}(P(v))=\checks{\alpha_1}$ is injective, and then $\alpha_3\alpha_2$ does not \checks{satisfy} the condition given in processes (a) and (b).
\end{itemize}
Therefore, $(\alpha_3\alpha_2, \alpha_2\alpha_1)$ is not a effective intersection relations,
this is a contradiction.

Furthermore, if there is an arrow $\alpha_{n+2}$ such that $\alpha_{n+2}\alpha_{n+1}\in \I$,
then we obtain that $(\alpha_{n+2}\alpha_{n+1},\ldots, \alpha_2\alpha_1)$ is also an ELIS of $P(v)$ and $\alpha_{n+2}\alpha_{n+1}$ is not a relation on any full relational cycle.
Since $\Q$ is a finite quiver, the lengths of all ELISs of $P(v)$ \checks{have} a supremum.
Thus, all ELISs of $P(v)$ have a supremum. Then, by Theorem \ref{thm:main} (2), $A$ is Gorenstein.
\end{proof}

\begin{example} \rm \label{exp:idA 1}
Let $A$ be the string algebra given in Example \ref{exp:ELIS}.
Then we have $\length(\ELIS(z^{-1}x^{-1}cd))=\infty$,
where $z^{-1}x^{-1}cd$ is the string corresponding to $P(2)$.
Thus, $\id A = \infty$.
\end{example}

\begin{example} \rm \label{exp:idA 2}
Let $A$ be the algebra given by the following quiver $\Q$
\begin{center}
\begin{figure}[H]
\centering
\begin{tikzpicture}
% relations i
\draw [green!50] [line width=2pt] ( 1, 1) to[out=135,in=-135] ( 1, 3);
\draw [green!50] [line width=2pt] ( 1, 3) to[out=135,in=-135] ( 1, 5);
\draw [green!50] [line width=2pt] (-1, 1) to[out=45,in=-45] (-1, 3);
\draw [green!50] [line width=2pt] (-1, 3) to[out=45,in=-45] (-1, 5);
% relations ii
\draw [ blue!50] [line width=2pt] (-1, 3) to[out=90,in=-180] (1,5);
\draw [ blue!50] [line width=2pt] ( 0, 4) to[out=90,in=135] (3.828,5);
\draw [  red!50] [line width=2pt] ( 1, 3) to[out=90,in= 0] (-1,5);
\draw [  red!50] [line width=2pt] ( 0, 4) to[out=90,in=45] (-3.828,5);
\draw ( 0, 0) node{$1$};
\draw (-1, 1) node{$2$};
\draw[->] (-1+0.2, 1-0.2) -- ( 0-0.2, 0+0.2); \draw ( 0.5, 0.5) node[right]{$a_{2'1}$};
\draw ( 1, 1) node{$2'$};
\draw[->] ( 1-0.2, 1-0.2) -- ( 0+0.2, 0+0.2); \draw (-0.5, 0.5) node[left]{$a_{21}$};
\draw ( 0, 2) node{$3$};
\draw[->] ( 0-0.2, 2-0.2) -- (-1+0.2, 1+0.2); \draw ( 0.5, 1.5) node[right]{$a_{32'}$};
\draw[->] ( 0+0.2, 2-0.2) -- ( 1-0.2, 1+0.2); \draw (-0.5, 1.5) node[left]{$a_{32}$};
\draw (-1, 3) node{$4$};
\draw[->] (-1+0.2, 3-0.2) -- ( 0-0.2, 2+0.2); \draw ( 0.5, 2.5) node[right]{$a_{4'3}$};
\draw ( 1, 3) node{$4'$};
\draw[->] ( 1-0.2, 3-0.2) -- ( 0+0.2, 2+0.2); \draw (-0.5, 2.5) node[left]{$a_{43}$};
\draw ( 0, 4) node{$5$};
\draw[->] ( 0-0.2, 4-0.2) -- (-1+0.2, 3+0.2); \draw ( 0.35, 3.6) node[below]{$a_{54'}$};
\draw[->] ( 0+0.2, 4-0.2) -- ( 1-0.2, 3+0.2); \draw (-0.35, 3.6) node[below]{$a_{54}$};
\draw (-1, 5) node{$6$};
\draw[->] (-1+0.2, 5-0.2) -- ( 0-0.2, 4+0.2); \draw ( 0.5, 4.5) node[right]{$a_{6'5}$};
\draw ( 1, 5) node{$6'$};
\draw[->] ( 1-0.2, 5-0.2) -- ( 0+0.2, 4+0.2); \draw (-0.5, 4.5) node[left]{$a_{65}$};
\draw (-2.414, 5) node{$7$};
\draw[->] (-2.414+0.2, 5) -- (-1-0.2, 5); \draw (-1.714, 4.914) node[above]{$a_{76}$};
\draw ( 2.414, 5) node{$7'$};
\draw[->] ( 2.414-0.2, 5) -- ( 1+0.2, 5); \draw (-3.114, 4.914) node[above]{$a_{87}$};
\draw (-3.828, 5) node{$8$};
\draw[->] (-3.828+0.2, 5) -- (-2.414-0.2, 5); \draw ( 1.714, 4.914) node[above]{$a_{7'6'}$};
\draw ( 3.828, 5) node{$8'$};
\draw[->] ( 3.828-0.2, 5) -- ( 2.414+0.2, 5); \draw ( 3.114, 4.914) node[above]{$a_{8'7'}$};
\end{tikzpicture}
%\caption{ }
%\label{fig:idA 2}
\end{figure}
\end{center}
with $\I = \langle {a_{87}a_{76}a_{65}}, {a_{65}a_{54'}},
{a_{8'7'}a_{7'6'}a_{6'5}}, {a_{6'5}a_{54}},
{a_{65}a_{54}}, {a_{43}a_{32}},
{a_{6'5}a_{54'}}, {a_{4'3}a_{32'}}\rangle$. Then $A$ is a string algebra and $\ELIS(P(5))$ have two elements
\begin{center}
$({a_{87}a_{76}a_{65}}, {a_{65}a_{54'}})$ and
$({a_{8'7'}a_{7'6'}a_{6'5}}, {a_{6'5}a_{54}})$
\end{center}
of lengths two. Thus,
$$\id P(5) = \length(\ELIS(P(5)))+1 = 2+1=3.$$
We can check all ELISs of indecomposable projective $A$-modules and obtain that $\id A = 3$. Thus, $A$ is Gorenstein.
\end{example}

%=========================================================

\section*{ }

\subsection*{\textbf{Authors' contributions}}
H. Zhang, D. Liu and Y.-Z. Liu contributed equally to this work.

\subsection*{\textbf{Funding}}
Houjun Zhang is supported by the National Natural Science Foundation of China (No. 12301051) and
the Natural Science Research Start-up Foundation
of Recruiting Talents of Nanjing University of Posts and Telecommunications (No. NY222092).

Dajun Liu is supported by the National Natural Science Foundation of China (No. 12101003) and
the Natural Science Foundation of Anhui province (No. 2108085QA07).

Yu-Zhe Liu is Supported by the National Natural Science Foundation of China (No. 12401042, 12171207),
Guizhou Provincial Basic Research Program (Natural Science) (Grant Nos. ZK[2025]085, ZK[2024]YiBan066),
and the Scientific Research Foundation of Guizhou University (Nos. [2023]16, [2022]53, and [2022]65).

%The work was supported by
%the National Natural Science Foundation of China (Nos. 12401042, 12301051, 12171207, and 12101003),
%the Natural Science Research Start-up Foundation of Recruiting Talents of Nanjing University of Posts and Telecommunications (No. NY222092),
%the Natural Science Foundation of Anhui province (No. 2108085QA07),
%Guizhou Provincial Basic Research Program (Natural Science) (Grant Nos. ZK[2025]085, ZK[2024]YiBan066)
%and the Scientific Research Foundation of Guizhou University (Nos. [2022]65 and [2022]53).

\subsection*{\textbf{Data Availability}} Data sharing not applicable to this article as no datasets were generated or analysed during the current study.

\subsection*{\textbf{Declarations}}

\subsection*{\textbf{Ethical Approval}} Not applicable.

\subsection*{\textbf{Acknowledgements}}
We are extremely grateful to referees for his/her meticulous review and the valuable comments and suggestions, which greatly improve the quality of our article.

%=========================================================
%\bibliographystyle{amsplain}
%\bibliography{class-tilt}

% 渐变色 [top color=purple, bottom color=orange]

%  \bibliographystyle{abbrv}  % 参考文献引用-abbrv格式
%  \bibliographystyle{unsrt} % 参考文献引用-unsrt格式

% \bibliographystyle{amsplain}
% \bibliography{referLiu20221024}

\end{document}